\documentclass{amsart}
\usepackage{amssymb,amsmath,amsfonts,amsthm}
\usepackage{graphicx}
\graphicspath{ {./images/} }

\usepackage{cleveref}

\usepackage{psfrag}
\usepackage{mathtools}
\usepackage{color}
\usepackage{todonotes}
\usepackage{enumitem}
\usepackage{mathrsfs}

\usepackage[OT2, T1]{fontenc}
\usepackage[russian, english]{babel}
\DeclareMathOperator\DD{\textrm{\foreignlanguage{russian}{D}}}

\theoremstyle{plain}
\newtheorem{main}{Theorem}

\newtheorem{assumption}{Assumption}

\newtheorem{theorem}{Theorem}[section]
\newtheorem{lemma}[theorem]{Lemma}
\newtheorem{proposition}[theorem]{Proposition}
\newtheorem{corollary}[theorem]{Corollary}
\theoremstyle{remark}
\newtheorem{remark}[theorem]{Remark}
\newtheorem{definition}{Definition}

\newcommand\numberthis{\addtocounter{equation}{1}\tag{\theequation}}

\providecommand{\dotminus}{\mathbin{\mathpalette\xdotminus\relax}}
\newcommand{\xdotminus}[2]{%
  \ooalign{\hidewidth$\vcenter{\hbox{$#1\dot{}$}}$\hidewidth\cr$#1-$\cr}%
}

\newcommand{\diam}{\operatorname{diam}}
\newcommand{\ess}{\operatorname{ess}}

               \def\cal{\mathcal}
           \def\ea{\end{array}}
          \def\ec{\end{center}}
     \def\ed{\end{description}}
        \def\ee{\end{equation}}
       \def\eea{\end{eqnarray}}
     \def\eeaa{\end{eqnarray*}}
 \def\et{\end{thebibliography}}
\def\bib{\bibitem}

\def\U{{\cal U}}

\def\q{\quad}

\def\Orb{{\rm Orb}}

\def\Lip{\operatorname{Lip}}

\def\cG{{\mathcal G}}
\def\cA{{\mathcal A}}

\def\cC{{\mathcal C}}

\def\cI{{\mathcal I}}

\def\cR{{\mathcal R}}
\def\cB{{\mathcal B}}

\def\cR{{\mathcal R}}

\def\II{{\mathbb{I}}}
\def\RR{{\mathbb R}}
\def\PP{{\mathbb P}}

\def\NN{{\mathbb N}}
\def\EE{{\mathbb E}}

\def\vp{{\varphi}}

\def\bM{{\bf{M}}}

\title[EVL and entry time]{Rare event process and entry times distribution for arbitrary null sets on compact manifolds}

\author{Fan Yang}
\date{\today}

\address{Department of Mathematics, University of Oklahoma, Norman, Oklahoma, USA.}
\email{fan.yang-2@ou.edu}

\begin{document}

\begin{abstract} We establish the general equivalence between rare event process for arbitrary continuous functions whose maximal values are achieved on non-trivial sets, and the entry times distribution for arbitrary measure zero sets. We then use it to show that the for differentiable maps on a compact Riemannian manifold that can be modeled by Young's towers, the rare event process  and the limiting entry times distribution both converge to compound Poisson distributions. A similar result is also obtained on Gibbs-Markov systems, for both cylinders and open sets. We also give explicit expressions for the parameters of the limiting distribution, and a simple criterion for the limiting distribution to be Poisson. This can be applied to a large family of continuous observables that achieve their maximum on a non-trivial set with zero measure. 
\end{abstract}

\maketitle

\tableofcontents

\section{Introduction}
The extreme value theory and its relation with entry/return times statistics have been a hot topic for the last decade. For a given potential function, one observes the occurrence of extreme phenomenon, when the observation of the potential along the underlying dynamical systems achieves very high value. When the maximal value of the potential is achieved at a generic point, the extreme value distribution is known to converge to one of the three limiting laws (Gumbel, Fr\'echet, or Weibull distribution, all of which are of the form $e^{-\tau}$), which agrees with the classical extreme value theory. We invite the reader to the book~\cite{FFFH} for more details. However, when the maximal value is achieved on a periodic point, one will pick up a point mass at the origin. This is because the periodic behavior will generate a cluster of exceedances, which will prevent generic points from entering its neighborhoods. It is then shown in~\cite{FFT13} that for non-periodic points, the total number of exceedances within a time scale suggested by Kac's theorem is Poissonian in limit, while for periodic points the limiting distribution is {\em compound Poisson}. In particular, the compound part is a geometric distribution, with parameter $\theta$ given by the portion of points that remains in the neighborhood under the iteration of $f^m$ where $m$ is the period. This is generally known as the P\'olya-Aeppli distribution, and the parameter $\theta$ is sometimes called the {\em extremal index}. 

For the limiting distribution of entry/return times, Pitskel~\cite{P91} proved that for Markov chains, the number of entries to cylinder neighborhoods around a generic point is Poissonian. This result is later generalized to systems with various kind of mixing properties, see for example~\cite{A06}. The same result holds for geometric balls when the map is modeled by Young's towers, which is proven by Collet, Chazottes~\cite{CC} for towers with exponential tails, and Haydn, Wasilewska~\cite{HW15}, Pene, Saussol~\cite{PS} for polynomial tails. In the case of periodic points, it is shown in~\cite{HV09} that the number of returns is close to a P\'olya-Aeppli distribution. 

It is not a coincidence that extreme value distributions and entry times distributions agree for  both non-periodic and periodic points. This is proven in~\cite{FFT10}, where the authors show that these distributions are equivalent if one considers potential functions that have certain symmetry and regularity near the maximal value. 

An important yet very difficult step forward is to study the entry/exceedance distribution for the neighborhoods of any measure zero set. One of the key motivations lies in the {\em shortest distance between different orbits}, which is studied in~\cite{BLR}. If one defines $\phi(x,y) = -\log \min_{0\le k\le n-1} \{d(f^kx, f^ky)\}$ which is the shortest distance between the orbits segments of $x$ and $y$ before time $n$, then $\phi$ can be seen as a potential function on the product system $f\times f:\bM\times\bM \to \bM\times\bM$ which achieves its maximal value (infinity) along the diagonal $\{(x,x):x\in\bM\}$. Then to study the distribution property of $\phi$, one is forced to look at the entry/return times to the diagonal under the product system. Another motivation is given in~\cite{FFRS}, where the authors study the extreme value distribution near a Cantor set. 

One of the most important advances in this direction is  in~\cite{FFM}, where it is shown that the {\em marked rare event point process} (i.e., one considers not only the number of exceedances, but also the spatial position where such exceedances happen) will converge to the compound Poisson process with intensity $\theta$ and multiplicity d.f. $\pi$, under the assumption that: 
\begin{enumerate}
\item the thresholds $\{u_n\}$ are taken such that the measure of $\{X_0> u_n\}$ is of order $1/n$; here $X_0 = \varphi$ is the potential function (below we will refer to it as the observable);
\item there is a normalizing sequence $\{a_n\}$ and $\theta\in [0,1]$, a probability distribution $\pi(x)$ such that 
$$
\lim_n\frac{\PP(R_{p,0}(u_n, x/a_n))}{\PP(U_mn)} = \theta(1-\pi(x));
$$ 
\item two technical conditions $\DD_p(u_n)^*$ and $\DD_p'(u_n)^*$ hold;
\item the Condition $ULC_p(u_n)$ (Unlikely Long Clusters).
\end{enumerate}
More importantly, the conditions $\DD_p(u_n)^*$ and $\DD_p'(u_n)^*$ can be verified if one assumes that the system has decay of correlations against $L^1$ functions. However, this is known to be a very strong assumption as it implies the decay of correlations against $L^\infty$ functions at exponential speed. In the meantime, it is unclear when the intensity, $\theta$  and the distribution, $\pi(x)$ exist.

A more recent breakthrough is obtained in~\cite{HV19}. In this paper, the authors establish the existence of the parameters $\theta$ and $\{\lambda_\ell\}$ for the compound Poisson distribution using the short return probabilities $\{\hat{\alpha}_\ell\}$; they also prove the convergence of the entry times distribution using a{\em compound binomial} approximation theorem. One of the key ingredients in the proof is the desynchronization of the neighborhoods $U_n$ with the cut-off of the short return time $K$ (previously, the short return time depends on $n$. See equations~\eqref{e.3} and~\eqref{e.beta}). This allows them to easily show convergence without worrying about the meaning of a `short' return. Then, they consider a family of systems that are `mostly' hyperbolic\footnote{It is likely that such maps are, in fact, modeled by Young's towers with polynomial tails; see~\cite{AA} and~\cite{AL}.}, in the sense that:
\begin{enumerate}[label=(\alph*)]
\item the stable and unstable disks are globally defined;
\item the contraction/expansion/distortion on such disks are `good' except on a set with small measure; 
\item the measure can be globally decomposed into conditional measures along the unstable disks;
\item the system has polynomial decay of correlations.
\end{enumerate}

In this paper, we will consider both cylinders and open neighborhoods around an arbitrary null set. The main goal is to establish the convergence of the (unmarked) rare event process\footnote{Using the word `process' may be a slight exaggeration, as we will only show the convergence of the limiting distribution instead of the convergence of entire process. However, we believe that such convergence can be obtained by modifying the compound binomial approximation theorem in~\cite{HV19} (i.e., show the approximation by the compound binomial {\bf process}) following the work of~\cite{FFT13,FFM}, which will probably require a standalone paper.} to the compound Poisson distribution, for maps that are either Gibbs-Markov or  modeled by Young's towers. Note that we do not assume how the measure of such neighborhoods approach zero, nor do we impose any condition such as  $ULC(u_n)$.
Following the work of~\cite{HV19}, the parameters $\{\lambda_\ell\}$ will be determined explicitly by the short return probabilities of such neighborhoods. We will demonstrate how to control the error term in the compound binomially approximation theorem (which are, unsurprisingly, very similar to the conditions $\DD_p^*$ and $\DD_p'^*$ in~\cite{FFM}), using either $\phi$-mixing or decay of correlations against $L^\infty$ functions, both at polynomial speed. 

We also obtain several approximation results under general settings (Lemma~\ref{l.5.5}, \ref{l.5.6}) along the way, which allows one to approximate open neighborhoods with cylinders.
As a corollary, we provide an easy-to-check criterion for the limiting distribution to be Poisson. We also show that for potential functions achieving their maximal value on a null set, the extreme value distribution converges to $e^{-\alpha_1\tau}$ (with $\alpha_1$ being the extremal index). 

The secondary products of our proof are Proposition~\ref{p.1} and Proposition~\ref{p.2}, where we show that whether or not one synchronizes $K$ with $n$ will not affect the parameter of the limiting distribution. This, in particular, proves that the $\alpha_1$ defined by~\eqref{e.4} below is indeed the extremal index studied in~\cite{FFRS}.

We do not aim to provide specific examples in this paper, as it has been shown in~\cite{AFLV} that every system with sufficient decay of correlations must admits Young's tower. On the other hand, computing the parameters are usually lengthy (see, for instance, those examples in~\cite{GHV,HV19,FFM,FFRS}), and will be carried out in a standalone paper.

\section{Statement of results}\label{s.2}

A random variable $W$ is {\em compound Poisson distributed}, if there exists i.i.d. random variables $Z_j,j=1,2,\ldots$ taking value in positive integers, and an independent Poisson distributed random variable $P$, such that $W=\sum_{j=1}^{P}Z_j$. In other words, the number of occurrences within each time interval can be partitioned into independent clusters, whose total number follows a Poisson distribution while the number of occurrences within each cluster is distributed according to $Z_j's$. If we set $\lambda_\ell=\PP(Z_j=\ell),\ell=1,2,\ldots$, then we say that $W$ is a compound Poisson distributed for the parameters $\{\lambda_\ell\}$. More details on the compound Poisson distribution will be given in Section~\ref{s.3}.

Throughout this paper, unless otherwise specified, we will assume that $(\bM, \cB, \mu, f)$ is a measure preserving systems with $\bM$ a compact Riemann manifold, $f:\bM\to \bM$ a differentiable map, $\cB$ the Borel $\sigma$-algebra and $\mu$ an $f$-invariant probability measure. We will frequently write $\PP=\mu$ when we interpret $\mu(A)$ as the probability of the event $A$. 

We take a continuous observable (potential function)
$$\vp:\bM\to\RR \cup \{\pm\infty\},$$
such that the maximal value of $f$ (which could be positive infinite) is achieved on a $\mu$  measure zero closed set $\Lambda$, and consider the process generated by the dynamics of $f$ and the observable $\varphi:$
$$
X_0=\varphi,\,X_1=\varphi\circ f,\,\ldots, X_k = \varphi\circ f^k
,\ldots.$$
Let $\{u_n\}$ be a non-decreasing sequence of real numbers and $\{w_n\}$ a non-decreasing sequence of integers with $u_n\to\sup f$ and $w_n\to\infty$, such that 
\begin{equation}\label{e.1}
w_n\PP(X_0>u_n)\to \tau \in \RR^+ \mbox{ as }n\to\infty
\end{equation}
for some positive real number $\tau$. We will think of $u_n$ as a sequence of thresholds, and the event $\{X_k>u_n\}$
marks an exceedance above the threshold $u_n$. Also denote by $U_n$ the open set
$$
U_n: = \{X_0>u_n\}.
$$

We are interested in the total number of such exceedances before time $N$. To this end, we define, for integers $n$ and $N$,
$$
\xi^N_{u_n}(x) = \sum_{k=0}^{N-1} \II_{\{X_k>u_n\}}(x),
$$
where $\II_U$ is the indicator function of the set $U$.
This is known as the {\em rare event process} in~\cite{FFT13}, under the special case $w_n = n$.

To characterize the limiting distribution of $\xi^N_{u_n}$ as $n\to\infty$ we first observe that since $\{u_n\}$ is non-decreasing and $f$ is continuous, we have $U_n\subset U_{n-1}$, and
$$
\cap_n U_n = \Lambda.
$$
It then follows that $\mu(U_n)\searrow 0=\mu(\Lambda).$ Furthermore, \eqref{e.1} means that the measure of $U_n$ is of the order $\tau/w_n$.  

To state the parameters of the compound Poisson distribution, we assume that the following limits  exist for $K$ large enough and every $\ell\ge 1:$
\begin{equation}\label{e.3}
\hat\alpha_\ell = \lim_{K\to\infty}\lim_{n\to\infty}\mu_{U_n}(\tau^{\ell-1}_{U_n}\le K),
\end{equation}
where $\mu_{U_n}(\tau^{\ell-1}_{U_n}\le K)$ is the conditional probability of having at least $(\ell-1)$ returns to $U_n$ before time $K$.  We will see later that one only need to assume that the limit in $n$ exists, since $\hat\alpha(K) := \lim_{n\to\infty}\mu_{U_n}(\tau^{\ell-1}_{U_n}\le K)$ is monotonic in $K$. See the discussion in Section~\ref{ss.3.1}.

Then we put for every integer $\ell >0$ and $K>0$, 
\begin{equation}\label{e.2}
\lambda_\ell(K,U_n) = \frac{\PP(\sum_{i=0}^{2K}\II_{U_n}\circ f^i=\ell)}{\PP(\sum_{i=0}^{2K}\II_{U_n}\circ f^i\ge 1)}.
\end{equation}
In other words, $\lambda_\ell(K,U_n)$ is, conditioned on having an entry to the set $U_n$, the probability to have precisely $\ell$ entries in a cluster with length $2K+1$.

We will see later that the existence of the limits defining $\hat\alpha_{\ell}$ implies the existence of the following limits:
\begin{equation}\label{e.3'}
\lambda_\ell= \lim\limits_{K\to+\infty}\lim\limits_{n\to\infty}\lambda_\ell(K,U_n),
\end{equation}
and
\begin{equation}\label{e.4}
\alpha_1 = \lim_{K\to\infty}\lim_{n\to\infty}\mu_{U_n}(\tau_{U_n}>K).
\end{equation}
The real number $\alpha_1\in(0,1)$ is generally known as the {\em extremal index} (EI). See Freitas et al~\cite{FFT13}.

More importantly, assuming the existence of $\{\hat\alpha_{\ell}\}$, we will see that $\{\lambda_\ell\}$ satisfies $\sum_{\ell}\lambda_\ell=1$ (thus can be realized as the distribution of some random variable $X_0$), and can be explicitly determined using $\{\hat{\alpha}_{\ell}\}$. The relation between these sequences can be found in Section~\ref{ss.3.1}, in particular, Theorem~\ref{t.1}. 

Next, we turn our attention to the nested sequence $\{U_n\}$. In the most general setting, the geometry of the set $U_n$ can be quite bizarre. To deal with this issue, we will make the following assumption on the shape of $U_n$. For each $r_n>0$, we approximate $U_n$ by two open sets (`o' and `i' stand for `outer' and `inner'):
$$
U^o_n = \bigcup_{x\in U_n} B_{r_n}(x), \mbox{\,\, and \,\,} U^i_n = U_n\setminus \left(\bigcup_{x\in \partial U_n}\overline{B_{r_n}(x)}\right).
$$
It is easy to see that 
$$
\overline{U^i_n}\subset U_n\mbox{ and } \overline{U_n}\subset  U^o_n,
$$
with 
$$
d(\overline{U^i_n},(U_n)^c)\ge r_n, \mbox{ and \,}d(\overline{U_n},(U^o_n)^c)\ge r_n. 
$$
The following assumption requires $U_n$ to be {\em well approximable} by $U^{i/o}_n$. 
\begin{assumption}\label{a.1}
There exists a positive, decreasing sequence of real numbers $\{r_n\}$ with $r_n\to0$ (whose rate will be specified later, see Theorem~\ref{m.5}), such that
\begin{equation}\label{e.5}
\mu\left(U^o_n\setminus U^i_n\right) = o(1)\mu(U_n).
\end{equation}
\end{assumption}
Here $o(1)$ means the term goes to zero under the limit $n\to\infty$. This also applies to the rest of the paper.

We will also impose an assumption on the topological boundary of $U_n$. 
\begin{assumption}\label{a.2}
The sets $U_n$ have `small boundaries', in the sense that for $r$ small enough (but doesn't need to be too small, depending on $n$), $\mu(B_r(U_n)) = \mu(U_n)+F(r)$ where $B_r(U_n) = \bigcup_{x\in U_n} B_{r}(x)$, and $F(r)$ is a function of $r$ with $F(r)\to 0 $ as $r\to 0$ (with certain rate that will be specified later, see Theorem~\ref{m.4} and~\ref{m.5}).
\end{assumption}

Next, we have to assume that the set $\{U_n\}$ consists mainly of `good points', in the sense that the tail of the tower has small measure in $U_n$. This assumption is more technical and as a result, the precise statement will be postponed to Section~\ref{s.5} (See the statement of Theorem~\ref{m.5}).

\begin{assumption}\label{a.3}\footnote{A similar condition is verified for geometric balls in~\cite{CC}, see in particular the appendix there. The proof uses the Besicovitch covering lemma, which clearly does not hold for arbitrary open sets. Therefore we state it as a technical assumption. }
There exists $K_0>0$ and $p''> 1$, such that for every $n$ large enough and every $K_0<k<w_n$, there is $0\le s(k)\ll  k/2$, such that the set (for the precise definition, see~\eqref{e.tilde}):
\begin{align*}
\tilde{\Omega}_i:= \{x\in\Omega_{0,i}: \mbox{ the last visit to }\Omega_0\mbox{ before time }k \mbox{ is in } \Omega_{0,m} \mbox{ with } R_m<s(k)\}
\end{align*}
satisfies
$$G(k):=\frac{\sum_i\sum\limits_{j=0}^{R_i}\mu_0(f^{-j}U_n\cap (\Omega_{0,i}\setminus \tilde{\Omega}_i))}{\mu(U_n)}\le C k^{-p''}.$$  
\end{assumption}

Finally, if $f$ is invertible, we will make the following additional assumption on the conditional measure of $U_n$.

\begin{assumption}\label{a.4}\footnote{This assumption can be weakened so that $\mu_{\gamma^u}(f^{-b}U_n\cap\Omega_0)\le C\mu(U_n)$ holds for all $\gamma$ except on a sequence of sets whose measures (with respect to the transversal measure on $\Gamma^s$) are small comparing to the measure of $U_n$. One only need to slightly modify the proof in Section~\ref{s.6.2}.}
There exists $C>0$, such that for each $0\le b\le s(1/\mu(U_n))$ and $\gamma^u\in\Gamma^u$, we have 
$$
\mu_{\gamma^u}(f^{-b}U_n\cap\Omega_0)\le C\mu(U_n).
$$
for $n$ large enough. Here $\Omega_0$ is the base of the tower, and $\mu_{\gamma^u}$ are the conditional measures of $\mu_0 = \mu|_{\Omega_0}$ along leaves in $\Gamma^u$ (the precise definition of $\mu_0$ and $\Gamma^u$ are in Section~\ref{ss.3.3}).
\end{assumption}

Note that the assumption holds trivially for those $\gamma^u\in\Gamma^u$ that do not intersect with $f^{-b}(U_n)$.

With that we are ready to state the main theorem of this article.

\begin{main}\label{m.2}
Assume that $f:\bM\to \bM$ is a $C^{1+\alpha}$ non-invertible map that can be modeled by Young's towers with summable tail. Let $\varphi:\bM\to\RR\cup\{\pm\infty\}$ be a continuous observable, achieving its maximum on a closed set $\Lambda$ with $\mu(\Lambda)=0$. Assume that there exists a sequence of thresholds $\{u_n\}$ such that~\eqref{e.1} is satisfied, and the corresponding sets $U_n$ satisfy Assumption 1 and 2, with $\sum_{\ell=1}^{\infty}\ell\hat\alpha_{\ell}<\infty$.

Suppose one of the following two assumptions holds:\\
\noindent (1) either the tower is defined using the first return map, and $U_n\subset \Omega_0$ for $n$ large enough;\\
\noindent (2) or Assumption~\ref{a.3} holds, and the decay rate satisfies  $\cC(k)=o(k^{-1})$.

Then we have
$$
\PP(\xi^{w_n}_{u_n} = k)\to m(\{k\})
$$
as $n\to\infty$, where $m$ is the compound Poisson distribution for the parameters $\{\lambda_\ell\}$.
\end{main}

The previous theorem has a similar formulation in the invertible case:

\begin{main}\label{m.1}
Assume that $f:\bM\to \bM$ is a $C^{1+\alpha}$  (local) diffeomorphism that can be modeled by Young's towers, with decay rate $\cC(k)=o(k^{-1})$. Let $\varphi:\bM\to\RR\cup\{\pm\infty\}$ be a continuous observable, achieving its maximum on a closed set $\Lambda$ with $\mu(\Lambda)=0$. Assume that there exists a sequence of thresholds $\{u_n\}$ such that~\eqref{e.1} is satisfied, and the corresponding sets $U_n$ satisfy Assumption 1 to 4, with $\sum_{\ell=1}^{\infty}\ell\hat\alpha_{\ell}<\infty$.

Then the rare event process $\xi^N_{u_n} = \sum_{k=0}^{N-1} \II_{\{X_k>u_n\}}$ satisfies
$$
\PP(\xi^{w_n}_{u_n} = k)\to m(\{k\})
$$
as $n\to\infty$, where $m$ is the compound Poisson distribution for the parameters $\{\alpha_1\tau\lambda_\ell\}$.
\end{main}

\begin{remark}
In both theorems, the assumption on the continuity of $\varphi$ can be weakened. One only need $\varphi$ to be upper semi-continuous, and take $U_n$ to be the closed set $\{X_0 \ge u_n\}$ or its interior. The proof applies without any modification. In fact, the proof below does not depend on whether $U_n$ are open or not.
\end{remark}

As the first corollary, we give a simple criterion for the limiting distribution to be indeed Poisson. For any measurable set $U\subset \bM$, we define the {\em periodic} of $U$, denoted by $\pi(U)$, as:
$$
\pi(U) = \min\{k>0: f^{-k}U\cap U\ne\emptyset\}.
$$
This can be seen as the {\em first time} that some point in $U$ returns to $U$. We also define the {\em essential periodic}\footnote{The period $\pi(\cdot)$ has been studied extensively in a series of papers (see for example,~\cite{STV} for the asymptotic behavior,~\cite{A13} for the fluctuation and~\cite{HY} for its relation with the local escape rate. However, as far as the author is aware, the essential period $\pi_{\ess}(\cdot)$ has not been previously studied.)} for a positive measure set $U$ to be 
$$
\pi_{\ess}(U) = \min\{k>0: \mu(f^{-k}U\cap U)>0\}.
$$
Clearly one has $\pi(U)\le \pi_{\ess}(U)$. On the other hand, $\mu$ is supported on the entire manifold $\bM$ and $U$ is open, then we have  $\pi(U)= \pi_{\ess}(U)$, as the nonempty intersection picked up by $\pi(U)$ must be an open set with positive measure.

\begin{corollary}\label{c.1}
Assume that the nested sequence $\{U_n\}$ satisfies 
$\pi_{\ess}(U_n)\to\infty$. Then the parameters $\alpha_1$ and $\{\lambda_\ell\}$ exist and satisfy $\alpha_1 = 1 = \lambda_1$, $\lambda_\ell = 0$ for $\ell\ge 2$. Furthermore, if the assumptions of either Theorem~\ref{m.2} or~\ref{m.1} holds, then the rare even process $\xi_{u_n}$ converges to a Poisson distribution with parameter $\tau$.

In particular, if $\Lambda =\cap_n U_n$ is contained in a fundamental domain of $f$ and does not contain periodic point, then the rare event process $\xi_{u_n}$ converges to a Poisson distribution with parameter $\tau$.

\end{corollary}
In particular, if $x$ is a non-periodic point then it is easy to see that $\pi(B_r(x))\to\infty$ as $r\to 0$. We then recover the classical result on the Poisson distribution for metric balls at non-periodic points. We would also like to point out that a similar condition is observed by Freitas et al in~\cite[Theorem 3.2]{FFRS} for interval maps and Cantor sets, when ``the dynamics considered is not compatible with the self-similar structure of the maximal set''.

The second corollary deals with the rare event distribution for the process $\{X_k\}$.  A similar result is obtained in a recent work by Freitas et al in~\cite{FFRS}, assuming two technical conditions, namely $\DD$ and $\DD'$, hold.\footnote{Similar to the conditions $\DD_p(u_n)^*$ and $\DD_p'(u_n)^*$ mentioned earlier, such conditions can be checked if one has decay of correlations against all $L^1$ observables. However, this assumption does not hold for Young's towers with sub-exponential tails. See the discussion in Remark~\ref{r.3}.}
 
\begin{corollary}
Assume that the conditions of Theorem~\ref{m.2} or~\ref{m.1} hold. Then the extremal value process
$$
M_n=\max\{X_k,k=0,\ldots,n-1\}
$$
satisfies
$$
\PP(M_{w_n}\le u_n)\to e^{-\alpha_1\tau}
$$
as $n\to\infty$. In particular, if $\pi_{\ess}(U_n)\to\infty$ then the limiting distribution is $e^{-\tau}$.
 
\end{corollary}

This corollary easily follows from the observation that $\{M_{w_n}\le u_n\}=\{\xi^{w_n}_{u_n} = 0\}$, and for a compound Poisson distribution $m$ with parameters $\{\alpha_{1}\tau\lambda_\ell\}$, $m(\{0\})=\PP(P=0)=e^{-\alpha_{1}\tau}$ where $P$ is the Poisson part of $m$. See the discussion on the properties of compound Poisson distribution in Section~\ref{ss.3.2}.

This paper is organized in the following way: in Section~\ref{s.3} we collect some existing results on the return and entry times to an arbitrary null set $\Lambda$, and establish the existence of the parameters $\{\lambda_\ell\}$ and $\alpha_1$. We will also introduce an abstract compound binomial approximation theorem which will be the main tool to show convergence to a compound Poisson process.

In Section~\ref{s.4}, we will establish the general equivalence between rare event process and entry times, thus converting the limiting distribution of rare event process for the observable $\vp$ to the limiting entry times distribution of the set $\Lambda$, on which $\vp$ achieves its maximum. The novelty here is that we do not assume the measure of the sets $\{X_0> u_n\}$ to be of order $\tau/n$. This is done in Theorem~\ref{m.3}.

Then in Section~\ref{s.6} and~\ref{s.5}, we prove the convergence of the entry times distribution to the compound Poisson process, for non-invertible and invertible systems respectively. To make the paper more interesting, we will use completely different techniques for these two cases: in the case of non-invertible maps, we will prove the convergence to the compound Poisson distribution for the induced system using $\phi$-mixing property, then apply an inducing argument to extend the result to the original map. This yields an interesting theorem by itself (Theorem~\ref{m.4}), and will allow us to get rid of the very technical Assumption~\ref{a.3};
in the case of invertible maps, we will use fast decay of correlations which is used by~\cite{CC} and~\cite{HW15}. 

We would like to point out that in all the theorems in this paper, we do not assume the measure $\mu$ to be the SRB measure (in the invertible case) or the absolutely continuous invariant probability (in the non-invertible case). As is shown in~\cite{PSZ} and several recent papers, Young's tower usually support many interesting measures other than the SRB measure.  Among them are the equilibrium states of geometric potentials, and sometimes the measure of maximal entropy, where our results can be applied.


\section{Preliminaries}\label{s.3}
In this section, we will introduce several notations that will be used through out the paper. Most importantly, we will introduce the short return and entry times on a sequence of nested sets, and deal with the existence of $\lambda_\ell$'s defined by~\eqref{e.3'}. Then we will state a compound binomial approximation theorem developed in~\cite{HV19}, which will enable us to show the convergence to the compound Poisson distribution. The last subsection contains the general definition of Young's towers for both invertible and non-invertible maps.

\subsection{Return and entry times on a sequence of nested sets}\label{ss.3.1}
In this section we recall the general results on the number of entries to an arbitrary null set $\Lambda$ within a cluster. For this purpose, we write, for any subset $U\subset \bM$, 
$$
\tau_U(x)=\min\{j\ge 1: f^j(x)\in U\}
$$
the first entry time to the set $U$. Then $\tau_U|_U$ is the first return time for points in $U$. Higher order entry times can be defined recursively:
$$
\tau^1_U=\tau_U, \mbox{ and }\tau^j_U(x) = \tau^{j-1}_U(x) + \tau_U(f^{\tau^{j-1}_U}(x)).
$$
For simplicity, we write $\tau^0_U = 0$ on $U$.

Given a sequence of nested sets $U_n,n=1,2,\ldots$ with $U_{n+1}\subset U_n$, $\cap_n U_n=\Lambda$ and $\mu(U_n)\to 0$, we will fix a large integer $K>0$ (which will be sent to infinity later), and assume that the limit 
$$
\hat\alpha_\ell(K) = \lim_{n\to\infty}\mu_{U_n}(\tau^{\ell-1}_{U_n}\le K)
$$
exists for $K$ sufficiently large and for every $\ell\in\NN$. By definition
$\hat\alpha_\ell(K)\ge \hat\alpha_{\ell+1}(K)$ for all $\ell$, and $\hat\alpha_1(K)=1$ due to our choice of $\tau^0$. Also note that $\hat\alpha_\ell(K)$ is non-decreasing in $K$ for every $\ell$. As a result, we have for every $\ell\ge 1$:
\begin{equation}\label{e.3.0}
\hat\alpha_\ell = \lim_{K\to\infty}\hat\alpha_\ell(K) \mbox{ exists for every }\ell, \mbox{ and } \hat\alpha_1=1,\hat\alpha_\ell\ge \hat\alpha_{\ell+1}.
\end{equation}

Note that in the definition of $\hat{\alpha}$, the cut-off for the short return time $K$ does not depend on the set $U_n$. 
Another way to study the short return properties for the nested sequence $U_n$ is to look at
\begin{equation}\label{e.beta}
\beta_\ell = \lim_{n\to\infty}\mu_{U_n}(\tau_{U_n}^{\ell-1}\le s_n)
\end{equation}
for some increasing sequence of integers $\{s_n\}$, with $s_n\mu(U_n)\to 0$ as $n\to\infty$. In other words, one can synchronize $K$ and $n$ in the same limit. This is the approach taken by Freitas et al in~\cite{FFRS}. However, we will see later in Proposition~\ref{p.1} and~\ref{p.2} that under our settings, we have $\beta_\ell=\hat{\alpha}_{\ell}$, while the latter is significantly easier to use (also potentially easier for numerical simulation). 

To demonstrate the power of desynchronizing $K$ from $n$, recall that for any set $U$, the  essential periodic of $U$ is given by:
$$
\pi_{\ess}(U) = \min\{k>0: \mu(f^{-k}U\cap U)>0\}.
$$
Then the following lemma can be easily verified using the definition of $\hat{\alpha}$:
\begin{lemma}\label{l.3.1}
Let $U_n$ be a sequence of nested sets. Assume that $\pi_{\ess}(U_n)\to \infty$ as $n\to\infty$, then $\hat{\alpha}_\ell$ exists and equals zero for all $\ell\ge 2$. 
\end{lemma}
\begin{proof}
For each $K$, one can take $n_0$ large enough such that $\pi_{\ess}(U_n)> K$ for all $n > n_0$. Then for $\ell \ge 2$, 
$$
\mu_{U_n}(\tau^{\ell-1}_{U_n}\le K)  \le \mu_{U_n} \left(\bigcup_{k=0}^K f^{-k}U_n\cap U_n\right) = 0
$$
since all the intersections have zero measure.
\end{proof}
Note that the converse of this lemma does not hold. Also note that the similar result for $\beta_\ell$ will require information on the rate at which $\pi(U_n)$ or $\pi_{\ess}(U_n)\to\infty$.  See~\cite{HY} for more detail.

Now let us come back to the properties of $\hat{\alpha}_{\ell}$. We assume that the limit
$$
p^\ell_i =\lim_{n\to\infty}\mu_{U_n}(\tau^{\ell-1}_{U_n}=i)
$$
exists for every $i\ge0,\ell\ge 1$. This is the limit of the conditional probability of the level sets of the $\ell$th return time $\tau^\ell_{U_n}$. Then it is shown in~\cite{HV19} that the following relation holds between $\{\hat\alpha_\ell\}$ and $\{p^\ell_i\}$.

\begin{lemma}\cite[Lemma 1]{HV19}\label{l.1}
	For every $\ell\ge 2$, we have
	$$
	\hat\alpha_\ell = \sum_i p^\ell_i.
	$$
\end{lemma}

Note that $\hat\alpha_\ell(K)$ is the conditional probability to have at least  $\ell-1$ returns in a cluster with length $K$. If we consider the level set:
$$
\alpha_\ell(K)=\lim_{n\to\infty}\mu_{U_n}(\tau^{\ell-1}_{U_n}\le K<\tau^\ell_{U_n})
$$
and its limit
\begin{equation}\label{e.3.1}
\alpha_\ell = \lim_{K\to\infty}\alpha_\ell(K),
\end{equation}
then it is easy to see that $\alpha_\ell = \hat\alpha_\ell - \hat\alpha_{\ell+1}$ which, in particular, implies the existence of $\alpha_\ell$. It also follows from the previous lemma that
$$
\alpha_\ell = \sum_i(p^{\ell-1}_i - p^\ell_i)
$$
for $\ell \ge 2$. In the special case $\ell=1$, we have
\begin{equation}\label{e.3.2}
\alpha_1 = \lim_{K\to\infty}\lim_{n\to\infty}\mu_{U_n}(\tau_{U_n}>K) = 1-\sum_ip^2_i.
\end{equation}

To see the relation between $\{\alpha_\ell\}$ and $\{\lambda_k\}$ defined by~\eqref{e.2} and~\eqref{e.3'}, we put 
$$
Z^K_{n} = \sum_{i=0}^{2K}\II_{U_n}\circ f^i
$$
which counts the number of entries to $U_n$ in a cluster with length $2K$. Then $\alpha_\ell(2K) = \lim_n\mu_{U_n}(Z^K_{n}=\ell)$, and~\eqref{e.2} can be written as
$$
\lambda_\ell(K,U_n)=\PP(Z^K_{n}=\ell|Z^K_{n}>0)=\frac{\PP(Z^K_{n}=\ell)}{\PP(Z^K_{n}>0)}. 
$$
Let us also introduce the notation
$$
Z^{K,-}_n =  \sum_{i=0}^{K-1}\II_{U_n}\circ f^i,\,\mbox{ and\, } Z^{K,+}_n =  \sum_{i=K}^{2K}\II_{U_n}\circ f^i.
$$
Then $Z^K_n= Z^{K,-}_n+Z^{k,+}_n$. \eqref{e.3.1} then becomes
\begin{equation}\label{e.3.3}
\alpha_\ell=\lim_{K\to\infty}\lim_{n\to\infty}\PP(Z^{K,-}_n = \ell|U_n) = \lim_{K\to\infty}\lim_{n\to\infty}\PP(Z^{K,+}_n = \ell|f^K(U_n)),
\end{equation}
where the second equality follows from the invariance of $\mu$. Note that the same expression holds in the case $\ell=1$. 

Define $W^K_n=\sum_{i=0}^{K}\II_{U_n}\circ f^i=Z^{K,-}_n + \II_{U_n}\circ f^K$. Then it follows that 
$$
\alpha_\ell=\lim_{K\to\infty}\lim_{n\to\infty}\PP(W^{K}_n = \ell|U_n).
$$

The next lemma controls the probability to have a very long cluster of entries.
\begin{lemma}\cite[Lemma 2]{HV19}\label{l.2} Assume that the limits in~\eqref{e.3.0} and~\eqref{e.3.1} exist and satisfy $\sum_{\ell=1}^{\infty}\ell\hat\alpha_\ell<\infty$. Then for every $\eta>0$, there exists $K_0>0$ such that for all $K'>K\ge K_0$, we have
$$
\PP_{U_n}(W_n^{K'-K}\circ f^K>0)\le \eta,
$$
for all $n$ large enough (depending on $K$ and $K'$).
\end{lemma}

Finally, we give the relation between $\{\lambda_\ell\}$ and $\{\alpha_\ell\}$.

\begin{theorem}\cite[Theorem 2]{HV19}\label{t.1}
Assume that $U_n$ is a sequence of nested sets with $\mu(U_n)\to 0$. Assume that the limits in~\eqref{e.3.0} exist for $K$ large enough and every $\ell\ge 1$. Also assume that $\sum_{\ell=1}^{\infty}\ell\hat\alpha_\ell<\infty$.

Then
$$
\lambda_\ell=\frac{\alpha_\ell - \alpha_{\ell+1}}{\alpha_1},
$$ 
where $\alpha_\ell = \hat\alpha_\ell -\hat\alpha_{\ell+1}$. In particular, the limit defining $\lambda_k$ exists. Moreover, the average length of the cluster of entries satisfies
$$
\sum_{\ell=1}^{\infty}\ell\lambda_\ell = \frac{1}{\alpha_1}.
$$
\end{theorem}

Note that by~\eqref{e.3'}, $\lambda_\ell\ge 0$ as long as they exist. This in turn shows that $\{\alpha_{\ell}\}$ is a non-increasing sequence in $\ell$, which, surprisingly enough, cannot be easily seen from their definitions. We also have $\sum_\ell \lambda_\ell = 1$ due to the telescoping sum.

The following lemma is a byproduct of the proof of the previous theorem. Write $\II^i = \II_{U_n}\circ f^i$, we get:
\begin{lemma}
For every $\eta>0$, we have
\begin{align*}
&\left|\PP(Z^{K,-}_n=k,Z_n^{K,+}=\ell-k, \II^K=1)-\PP(Z_n^{K,-}=k',Z_n^{K,+}=\ell-k', \II^K=1)\right|\\
&\le \eta\mu(U_n)
\end{align*}
for all $0\le k,k'<\ell$, provided that $K$ and $n$ are large enough.

\end{lemma}

To conclude this section, we introduce the next lemma on the entry times (note that the probability below is NOT conditioned on $U_n$), which will be used to show the convergence of the parameters of the compound Poisson distribution:

\begin{lemma}\cite[Lemma 3]{HV19}\label{l.3} Under the assumptions of Theorem~\ref{t.1}, we have
$$
\lim_{K\to\infty}\lim_{n\to\infty}\frac{\PP(\tau_{U_n}\le K)}{K\mu(U_n)}=\alpha_1.
$$
\end{lemma}

\subsection{Compound Poisson distribution and a compound binomial approximation theorem}\label{ss.3.2}
Here we review the general properties of compound Poisson distributions and state the approximation theorem that was proven in~\cite{HV19}.

A probability measure $m$ on $\NN_0=\NN\cup\{0\}$ is compound Poisson distributed with parameters $\{s\lambda_\ell:\ell\ge 1\}$,  if the probability generating function $\vp_m$ is given by
$$
\vp_m(z)=\exp(\int_0^\infty(z^x-1)\,d\rho(x)),
$$
where $\rho$ is the measure on $\NN$ defined by $\rho = \sum_{\ell}s\lambda_\ell\delta_\ell$; here $\delta_\ell$ is the point mass at $\ell$. If we write $L=\sum_{\ell}s\lambda_\ell$, then $L^{-1}\rho$ becomes a probability measure. Let $P$ be a Poisson random variable with parameter $L$, and $Z_j,j=1,\ldots$ an i.i.d. sequence of random variables with 
$$
\PP(Z_j=\ell)=\lambda_\ell = L^{-1}s\lambda_\ell.
$$
Then the random variable $W=\sum_{j=1}^{P}Z_j$ is a compound Poisson distribution. We will refer to $P$ as the {\em Poisson part}, and $Z_j$ as the {\em compound part} of $W$. If we have in addition that $\sum_{\ell\ge 1}\lambda_\ell=1$ (which is the case in this paper due to Theorem~\ref{t.1} and the remark afterward), then $L=s$, and $\EE(W)=s\EE(Z_1)$. Moreover, we will see later (Remark~\ref{r.2}) that $s=\tau\alpha_{1}$, which is the desired parameter for Theorem~\ref{m.2} and~\ref{m.1}.

Just like the classical Poisson distribution can be approximated by binomial distributions, compound Poisson distribution can be approximated by {\em compound binomial distributions} with the same compound part. For this purpose, we take a large integer $N$, a parameter $s>0$ and put $p=s/N$. Let $Q$ be a binomially distributed random variable with parameters $(N,p)$, and define 
$$
W'=\sum_{j=1}^{Q}Z_j,
$$
where $Z_j's$ are i.i.d. random variables as before. $W'$ has generating function $\vp_{W'}(z)= (p(\vp_{Z_1}-1)+1)^N$, where $\vp_{Z_1}(z)=\sum_{\ell}z^\ell\lambda_\ell$ is the generating function of $Z_1.$ Note that as $N$ tends to infinity, $Q$ converges to a Poisson distribution with parameter $s$, thus $W'$ will converge to a compound Poisson distribution $W$ with parameters $\{s\lambda_\ell\}$. This can be easily proven by checking the convergence of the generating function.

The following theorem gives the convergence of a dependent, stationary $\{0,1\}$-valued process to a compound binomial distribution:

\begin{theorem}\cite[Theorem~3]{HV19}\label{t.2}
Let $\{X_n\}_{n\in\NN}$ be a stationary $\{0,1\}$-valued process and $W^N=\sum_{i=0}^NX_i$ for some large integer $N$. Let $K,\Delta$ be positive integers such that $\Delta(2K+1) < N$ and define 
$Z=\sum_{i=0}^{2K}X_i$, $W_a^b=\sum_{i=a}^bX_i$. Let $\tilde{m}$ be the compound binomial distribution measure where the binomial distribution has values $p=\PP(Z\ge 1)$ and $N'=N/(2K+1)$, and the compound part has probabilities $\lambda_\ell=\PP(Z=\ell)/p$. 

Then there exists a constant $C$, independent of $K$ and $\Delta$, such that
$$
|\PP(W^N=k)-\tilde{m}(\{k\})|\le C(N'(\cR_1+\cR_2)) + \Delta\PP(X_0=1),
$$
where
\begin{flalign*}
\cR_1=\sup\limits_{\substack{M\in[\Delta,N']\\q\in(0,N'-\Delta)}}\Bigg|\sum_{u=1}^{q-1}\Big(\PP\Big(Z=u\,\wedge\,& W^{M(2K+1)}_{\Delta(2K+1)}=q-u\Big)-&\\
&\PP(Z=u)\PP\left(W^{M(2K+1)}_{\Delta(2K+1)}=q-u\right)\Big)\Bigg|,&
\end{flalign*}
and
\begin{flalign*}
\cR_2=\sum_{n=2}^{\Delta}\PP&\left(Z\ge 1 \,\wedge\,Z\circ f^{(2K+1)n}\ge 1\right).&
\end{flalign*}
\end{theorem}

\begin{remark}\label{r.2}
If one takes a sequence of nested sets $\{U_n\}$ with $\mu(U_n)\to 0$, then the parameters of the binomial part are $p= \PP(\tau_{U_n}\le 2K)$ and $N' = \frac{\tau}{(2K+1)\mu(U_n)}$. Then as $n\to\infty$ then $K\to\infty$, the binomial part will converge to a Poisson distribution with parameter:
$$
s=\lim_K\lim_n pN' = \tau\lim_K\lim_n\frac{\PP(\tau_{U_n}\le 2K)}{(2K+1)\mu(U_n)}=\tau\alpha_{1},
$$
due to Lemma~\ref{l.3}. As a result, the parameters of the compound part will converge to $s\lambda_\ell = \tau\alpha_{1}\lambda_\ell$, as desired. 
\end{remark}

\begin{remark}
This theorem and its proof are similar to the abstract Poisson approximation theorem by Collet and Chazottes~\cite{CC}, where the error terms $\cR_1$ and $\cR_2$ are also similar to the error terms in the classical Chen-Stein method by Arratia el al~\cite{AGG}. A Chen-Stein method approach to the compound Poisson distribution is also under development by Gallo, Haydn and Vaienti~\cite{GHV}.

However, we would like to point out that the Chen-Stein method may not be suitable for invertible maps with Young's towers, due to the gap in both error terms being opened towards the past, making it difficult to apply the decay of correlations. 
\end{remark}

\begin{remark}\label{r.3}
The error terms $\cR_1$ and $\cR_2$ are similar to the conditions $D_p(u_n)^*$ and $D'_p(u_n)^*$ used by Freitas et al in~\cite{FFT13}. As we will see later, $\cR_1$ can be verified similar to $D_p$ using decay of correlations against $L^\infty$ functions. On the other hand, the proof of $D'_p$ in~\cite{FFT13} requires decay of correlations against $L^1$ functions, which does not hold for Young's towers with less than exponential tail. This is because decay of correlations against $L^1$ functions at summable rate implies the decay of correlations again all $L^\infty$ functions with exponential rate~\cite[Theorem B]{AFLV}. 

However, as we will see in later sections, the error term $\cR_2$ is very easy to verify due to the desynchronization between $K$ and $n$.

\end{remark}

\subsection{Young's towers}\label{ss.3.3}
Young's towers, also known as the Gibbs-Markov-Young structure, is first introduced by Young in~\cite{Y2} and~\cite{Y3} as a discrete time suspension over a countable Markov map. The base of the tower is constructed in a way such that every time a partition set returns, it will be mapped to the entire base, with well controlled hyperbolicity and distortion estimates. It turns out that the decay of correlations for the tower depends on the time it takes for points to return. The first paper,~\cite{Y2}, deals with the local diffeomorphisms on compact manifolds whereas the second paper,~\cite{Y3}, contains a more abstract setting for non-invertible systems. Below, we will discuss these two cases separately.

\subsubsection{The non-invertible case}
In this subsection we assume that  $f$ is a differentiable map of a
Riemannian manifold $M$. Assume that there is a subset $\Omega_0\subset M$ with the following properties:\\
(i) $\Omega_0$ is partitioned into disjoint sets $\Omega_{0,i}, i=1,2,\dots$ and
there is a {\em return time function} $R:\Omega_0\rightarrow\mathbb{N}$,  constant on
the partition elements $\Omega_{0,i}$,  such that $f^R$ maps $\Omega_{0,i}$ bijectively
to the entire set $\Omega_0$. We write $R_i=R|_{\Omega_{0,i}}$.\\
(ii) For $j=1,2,\dots, R_i-1$ put $\Omega_{j,i}=\{(x,j): x\in\Omega_{0,i}\}$
and define $\Omega =\bigcup_{i=1}^\infty\bigcup_{j=0}^{R_i-1}\Omega_{j,i}$. Note that $\{(x,0): x\in\Omega_{0,i}\}$ can be naturally identified with $\Omega_{0,i}$.
$\Omega$ is called the {\em Markov tower} for the map $f$. It has the associated partition
$\mathcal{A}=\{\Omega_{j,i}:\; 0\le j<R_i, i=1,2,\dots\}$ which typically is countably infinite.
The map $F: \Omega\to\Omega$ is given by
$$\
F(x,j)=\left\{\begin{array}{ll}(x,j+1) &\mbox{if } j<R_i-1\\
(Tx,0)&\mbox{if } j=R_i-1\end{array}\right.
$$
where we put $T=f^R$ for the induced map on $\Omega_0$. If we denote by $\pi_\Omega:\Omega\to \bM$, $\pi_\Omega((x,j))=f^j(x)$ then $\pi_\Omega$ semi-conjugates $F$ and $f$.\\
(iii) Non-uniformly expanding: there is $0<\kappa<1$ such that for all $x,y\in\Omega_{0,i}$, $d(Tx,Ty)>\kappa^{-1}d(x,y)$. Moreover, there is
$C>0$ such that $d(f^kx,f^ky)\le Cd(Tx, Ty)$ for all $x,y\in\Omega_{0,i}$ and $0\le k < R_i$.
\\
(iv) The {\em separation time function} $s(x,y)$,  $x,y\in\Omega_0$, is defined as the largest positive
$n$ so that $(f^R)^jx$ and  $(f^R)^jy$ lie in the same sub-partition elements
for $0\le j<n$, i.e. $(f^R)^jx, (f^R)^jy\in\Omega_{0,i_j}$
for some $i_0,i_1,\dots,i_{n-1}$ while $(f^R)^jx$ and $(f^R)^jy$ belong to different $\Omega_{0,i}$'s.
We extend the separation time function to all of $\Omega$
by putting $s(x,y)=s(F^{R-j}x,F^{R-j}y)$ for $x,y\in\Omega_{j,i}$.\\
(v) There is a finite given `reference' measure on $\Omega_0$ which can be lifted to $\Omega$ by $F$. We denote the measure on $\Omega_0$ by $\nu_0$ and the lifted measure by $\nu$, and assume that the Jacobian $JF=\frac{d(F^{-1}_*\nu)}{d\nu}$
is  H\"older continuous in the following sense:
there exists a $\lambda\in(0,1)$ so that
$$
\left|\frac{Jf^Rx}{Jf^Ry}-1\right|\le C_2 \lambda^{s(Tx,Ty)}
$$
for all $x,y\in \Omega_{0,i}$, $i=1,2,\dots$.

The reference measure on $\Omega_0$ is often taken to be the Riemannian volume restricted to $\Omega_0$. If the return time $R$ is integrable with respect to  $\nu_0$, i.e., $$\int_{\Omega_0} R\,d\nu_0<\infty,$$ then by~\cite[Theorem~1]{Y3}
, there exists an $F$-invariant probability measure
$\tilde{\mu}$ on $\Omega$ which is absolutely continuous with respect to $\nu$. Then the pushed forward measure $\mu = \pi_*\tilde{\mu}$ is a measure on $\bM$ which is absolutely continuous with respect to the Riemannian volume.

When the return time function $R$ is the first return time of $x$ to the base $\Omega_0$, i.e., $R(x) = \tau_{\Omega_0}(x)$, then we say that the tower is defined using the first return map. In this case, the semi-conjugacy $\pi$ indeed conjugates the tower with the real dynamics.

The set $\{x:(R(x)>k)\}$ is usually referred to as {\em the tail of the tower}. It has been shown in~\cite{Y3} that if $\nu_0(R>k)\le  Ck^{-p}$ for some $C>0$ and $p>1$, then the system has {decay of correlations} for H\"older (or Lipschitz) functions against $L^\infty$ functions at polynomial rate: let $C^{\gamma}$ be the space of $\gamma$-H\"older functions from $\bM$ to $\RR$; then for any functions $\phi\in C^{\gamma}$ and $\psi\in L^\infty$, we have
\begin{equation}\label{e.decay1}
\left|\int_\bM\phi\cdot\psi\circ f^k\,d\mu-\int_\bM\phi\,d\mu\int_\bM\psi\,d\mu\right|\le C\|\phi\|_\gamma\|\psi\|_{L^\infty} \cC(k),
\end{equation}
where $\cC(k)$ is a positive, decreasing sequence with $\cC(k)\to 0$ as $k\to\infty$, with rate depending on $\nu_0(R(x)>k)$.

\subsubsection{The invertible case}
Next we consider the invertible case. We refer the readers to~\cite{Y2} and~\cite{AA} for the precise definition. Roughly speaking, a (local) diffeomorphism $f$ is modeled by Young's towers if there exists a set $\Lambda$ and two continuous families $\Gamma^s = \{\gamma^s_x\}$ and $\Gamma^u=\{\gamma^u_x\}$ of smooth stable and unstable disks with $\dim\gamma^s+\dim\gamma^u=\dim \bM$, such that $\Lambda$ consists of points that are the (unique) transverse intersection of disks in  $\Gamma^s$ and $\Gamma^u$. Without loss of generality, we assume that the diameter of all the disks in $\Gamma^s$ and $\Gamma^u$ are between $1/2$ and $1$.

It is then assumed that there is a partition of $\Gamma^s=\cup_i\Gamma^s_i$. One should think of each $\Lambda_i$ as the `product' of the $\Omega_{0,i}$ with entire stable disks. If we denote by $\Lambda_i$ the intersection of disks in $\Gamma^s_i$ with disks in $\Gamma^u$, then $\{\Lambda_i\}$ is a partition of $\Lambda$. Consider the {\em return time function} $R$, which is a function that is constant on each  $\Lambda_i$ (thus $R$ is constant on every stable disk), such that $f^{R}(\Lambda_i)$ consists of entire $u$-disks intersecting with $\Lambda$. In particular, this means that $f^R$ has the Markov property:
$$
f^R(\gamma^s(x))\subset \gamma^s(f^R(x)), \mbox{ and } f^R(\gamma^u(x))\supset \gamma^u(f^R(x)).
$$

Similar to (iii) of the non-invertible towers, we assume that on unstable disks, $f$ is backward contracting at polynomial rate: 
\begin{equation}\label{e.exp}
\forall \gamma^u\in\Gamma^u, x,y\in\gamma^u, n\ge 0, \mbox{ we have }d(f^{-n}x, f^{-n}y )\le \frac{C}{n^\alpha}.
\end{equation}

Similarly, $f$ is forward contracting at polynomial rate along stable disks:\footnote{In most examples (such as those in~\cite{Y2},~\cite{AL} and~\cite{Y19}), the contracting rates along both stable and unstable disks are indeed exponential. This is because the measures in question are usually hyperbolic (i.e., all the Lyapunov exponents are non-zero), and the return map is defined using `hyperbolic times'. To be more precise, for $\eta\in(0,1)$, a positive integer $n$ is called a $(\eta,u)$-hyperbolic time of $x$, if for every $0\le k<n$, we have $\Pi_{j=k}^{n} \|Df^{-1}(f^j(x)|_{E^u})\|\le C \eta^{n-j}$. $(\eta,s)$-hyperbolic times are defined similarly using the forward iterations of $f$. 

Every hyperbolic measure has plenty of hyperbolic times for typical points of the measure, due to the Pliss lemma.  Also note that if $n$ is a hyperbolic time of $x$ and $m$ is a hyperbolic time of $f^n(x)$, then $n+m$ is a hyperbolic time of $x$. Therefore, one can ask the contracting estimate to hold for {\em every} $n\ge 0$ as long as all the return times $R(x)$ are hyperbolic times of $x$ (or if they are close to a hyperbolic time). In this case, the measure of the tail of the tower, $\nu(R(x)\ge k)$, coincides with the tail of the hyperbolic times $\nu(n(x)\ge k)$, where $n(x)$ is the first hyperbolic time along the orbit of $x$. We refer the readers to~\cite{AL} and~\cite{Y19} for more details on hyperbolic times and how to use them to construct Young's towers.} 
\begin{equation}\label{e.cont}
\forall \gamma^s\in\Gamma^s, x,y\in\gamma^s, n\ge 0, \mbox{ we have }d(f^{n}x, f^{n}y )\le \frac{C}{n^\alpha}.
\end{equation}
Note that such contracting/expanding rate only applies to disks in $\Gamma^s$ and $\Gamma^u$, which are usually only defined inside a very small open ball in $M$.

The separation function $s(x,y)$ is defined in a similar way as in the non-invertible case, with the extra assumption that $s(x,y)$ only depends on the stable disks that contain $x$ and $y$.
The reference measure $\nu$ is usually taken such that the conditional measures of $\nu$ are the restriction of the Riemannian volume on the unstable disks, which we denote by $\nu_{\gamma^u}$. Then it is assumed that the Jacobian of the return map,  $Jf^R|_{\gamma^u}$, is H\"older continuous: for every $\gamma^u\in\Gamma^u$ and $x,y\in\gamma^u$,
$$
\log\frac{Jf^R|_{\gamma^u}(x)}{Jf^R|_{\gamma^u}(y)}\le \beta^{s(f^R(x),f^R(y))}.
$$

We also need the Jacobian of the holonomy map along stable disks, denoted by $\Theta_{{\gamma^u}',\gamma^u}:{\gamma^u}'\cap\Lambda\to\gamma^u\cap\Lambda$, to be absolutely continuous with respect to the reference measure $\nu_{{\gamma^u}'}$.

It is shown in~\cite{Y2} and~\cite{AA} that under the above assumptions, if the return time function $R$ is integrable with respect to some $\nu_{\gamma^u}$, then there exists a measure $\mu_0$, supported on $\Omega_0$, whose conditional measures along $\gamma^u$ are absolutely continuous with respect to  $\nu_{\gamma^u}$.  $\mu_0$ can be lifted to a measure $\mu$ on the entire tower, which is an SRB measure. Moreover, the system has decay of correlations for H\"older functions against $L^\infty$ functions that are constant on stable disks: if $\nu_{\gamma^u}(R>k)\le C k^{-p}$ for some $\gamma^u\in\Gamma^u$, $C>0$ and $p>1$, then one has
\begin{equation}\label{e.decay2}
\left|\int_\bM\phi\cdot\psi\circ f^k\,d\mu-\int_\bM\phi\,d\mu\int_\bM\psi\,d\mu\right|\le C\|\phi\|_\gamma\|\psi\|_{L^\infty} \cC(k),
\end{equation}
for $\phi\in C^\gamma$ and $\psi\in L^\infty$ such that $\psi|_{\gamma^s}$ is constant for every $\gamma^s$. 
Here $\cC(k)$ is a positive, decreasing sequence with $\cC(k)\to 0$ as $k\to\infty$, with rate depending on $\nu_0(R(x)>k)$.
\footnote{In~\cite{AA} the decay of correlation is proven when $\phi$ and $\psi$ are both H\"older continuous. However, since the proof there uses the quotient along stable disks to obtain a non-invertible tower, where the decay of correlations is known for $\psi\in L^\infty$ by~\cite[Theorem 3]{Y3}, one can easily check that the same proof carries over to $L^\infty$ functions $\psi$ that are constant on stable disks.} The rate function $\cC(k)$ is of order $k^{-(p-1)}$ if $\nu_{\gamma^u}(R>k)\le Ck^{-p}$, and is (stretched) exponential if $\nu_{\gamma^u}(R>k)$ is (stretched) exponential.

\section{Equivalence of rare event process and entry times distribution}\label{s.4}

In this section we will establish the relation between rare event process and entry times distributions. Such relation was first discovered by Freitas et al in~\cite{FFT10} for rare event laws and first entry times distributions, in the case $w_n=n$ and $U = B_r(x)$.

Recall that $\{U_n\}$ is a sequence of nested sets whose measures satisfy~\eqref{e.1}, and $\xi^{N}_{u_n}= \sum_{k=0}^{N-1} \II_{\{X_k>u_n\}}$ is the rare event process defined with respect to $\{u_n\}$ and the process $X_j=\vp\circ f^j$. On the other hand, we define the {\em entry times distribution} of a set $U$ as 
\begin{equation}\label{e.4.1}
\zeta^N_U = \sum_{k=0}^{N-1}\II_{U}\circ f^k.
\end{equation}
The next general theorem states that the distribution of $\xi^{N}_{u_n}$ and $\zeta^N_{U_n}$ are the same:

\begin{main}\label{m.3}
For any measure preserving system $(\bM, \cB,\mu,f)$ and any continuous function $\vp:\bM\to\RR\cup\{\pm\infty\}$, let $\{u_n\}$, $\{w_n\}$ be two non-decreasing sequences such that~\eqref{e.1} holds for the process $X_j=\vp\circ f^j$. Then for the nested sets $U_n = \{X_0>u_n\}$, the following statements are equivalent:
\begin{enumerate}
	\item there exists a distribution $m$ such that $\PP(\xi^{w_n}_{u_n} = k)\to m(\{k\})$ as $n\to\infty$ for every $k$;
	\item there exists a distribution $m$ such that $\PP(\zeta^{{\tau}/{\mu(U_n)}}_{U_n} = k)\to m(\{k\})$ as $n\to\infty$ for every $k$.
\end{enumerate}
Here $\tau>0$ is given by~\eqref{e.1}. 
\end{main}
\begin{remark}
Note that in this theorem, we do not assume any type of mixing condition, nor do we need any regularity assumptions on $U_n$ such as Assumption~\ref{a.1}. Also note that the distribution $m$ depends implicitly on $\tau>0$.
\end{remark}
\begin{remark}
Due to the Kac's theorem, the average of the return time on any positive measure set $U$ is given by  $\frac{1}{\mu(U)}$. This coincides with the normalizing factor $\tau/\mu(U_n)$. 
\end{remark}

An an simple corollary, we obtain the equivalence between extremal value laws and first entry times distribution for any continuous observable:
\begin{corollary}\label{c.2}
Under the assumptions of Theorem~\ref{m.3}, the following statements are equivalent:
\begin{enumerate}
	\item the extremal value process $	M_n=\max\{X_k,k=0,\ldots,n-1\}$	satisfies \\$
	\PP(M_{w_n}\le u_n)\to G(\tau)$ for some function $G$;
	\item the first entry time $\tau_{U_n}$ satisfies $\PP\left(\tau_{U_n}>\frac{\tau}{\mu(U_n)}\right)\to G(\tau)$ for some function $G$.
\end{enumerate}
\end{corollary}
\begin{proof}[Proof of Theorem~\ref{m.3}]
From the definition of $X_j$ and $U_n$, we see that 
\begin{align*}
\{X_k>u_n\} = &\{\vp\circ f^k> u_n\}\\
=&f^{-k}\{\vp>u_n\}\\
=&f^{-k} U_n,
\end{align*}
which means
\begin{align*}
\xi^{N}_{u_n} = &\sum_{k=0}^{N-1} \II_{\{X_k>u_n\}} = \sum_{k=0}^{N-1} \II_{f^{-k} U_n} \\
= &\sum_{k=0}^{N-1} \II_{ U_n}\circ f^k \\ = &\,\zeta^N_{U_n}.\\
\end{align*}

To prove the theorem, it suffices to show that 
$$
\left|\PP(\xi^{w_n}_{u_n} = k)-\PP(\zeta^{{\tau}/{\mu(U_n)}}_{U_n} = k)\right|\to 0
$$ 
for each $k$, since then the convergence of either one of them to $m(\{k\})$ will imply the convergence of the other to the same limit.
To this end, we write 
$$
a_n = \min\{w_n,\tau/\mu(U_n)\} \mbox{ and }b_n = \max\{w_n,\tau/\mu(U_n)\},
$$
then we have
$$
\{\xi^{w_n}_{u_n} = k\}\dotminus\{\zeta^{\tau/\mu(U_n)}_{U_n}=k \}\subset \left\{\sum_{k=a_n}^{b_n-1}\II_{ U_n}\circ f^k \ge 1\right\},
$$
where $\dotminus$ is the symmetric difference.

It then follows that
\begin{align*}
&\left|\PP(\xi_{u_n}^{w_n} = k)-\PP(\zeta_{U_n}^{{\tau}/{\mu(U_n)}} = k)\right|\\
\le& \PP\left(\sum_{k=a_n}^{b_n-1}\II_{ U_n}\circ f^k \ge 1\right)\\
= &\mu\left(\bigcup_{k=a_n}^{b_n-1}f^{-k}U_n\right)\\
\le& (b_n-a_n)\mu(U_n)\\
=&\left|w_n-\frac{\tau}{\mu(U_n)}\right|\mu(U_n)\\
=&\left|w_n\mu(U_n)-\tau\right|\to 0,
\end{align*}
thanks to~\eqref{e.1}. This finishes the proof of Theorem~\ref{m.3}.
\end{proof}
\begin{proof}[Proof of Corollary~\ref{c.2}]
Note that 
\begin{align*}
\{M_{w_n}\le u_n\} =&\{X_j\le u_n \mbox{ for all } j=0,1,\ldots, w_n-1\}\\
= &\{\xi_{u_n}^{w_n}=0\}..
\end{align*}

On the other hand, 
$$
\left\{\tau_{U_n}>\frac{\tau}{\mu(U_n)}\right\} = \{\zeta_{U_n}^{\tau/\mu(U_n)}=0\}.
$$
So the corollary follows from Theorem~\ref{m.3} by taking $k=0$, and considering $m(\{0\})$ as a function of $\tau$.

\end{proof}
\begin{remark}
Corollary~\ref{c.2} is first obtained in~\cite{FFT10} for functions $\vp$ where the maximal value is achieved at a single point $x$, for the case $w_n=n$. Moreover, it is assumed that the function has certain regularity near $x$. It turns out that such regularity assumption will make the extremal value distribution to be either Gumbel, Fr\'echet, or Weibull distribution. See the book~\cite{FFFH} for more discussion on this topic.
\end{remark}

\section{Proof of Theorem~\ref{m.2} when $U_n\subset \Omega_0$}\label{s.6}
This Section contains the proof of Theorem~\ref{m.2}, under the additional assumption that $U_n\subset \Omega_0$ for $n$ large enough and that the tower is defined using the first return map. The general case will be dealt with in the next section.

In view of Theorem~\ref{m.3}, we only need to show that under the assumptions of Theorem~\ref{m.2}, the distribution of the entry times $\zeta_{U_n}$ converges to the compound Poisson distribution with parameters $\{\alpha_1\tau\lambda_\ell\}$. The proof is carried out in four steps: 
\begin{enumerate}
	\item first we show that for {\em $\phi$-mixing measures}, the entry times distribution for a union of cylinders can be approximated by a compound binomial distribution;
	\item furthermore, assume that the systems is {\em Gibbs-Markov}, we will show the convergence to the compound Poisson distribution, for a sequence of nested cylinder sets; this step yields a theorem that is interesting in itself (Theorem~\ref{m.4});
	\item then we will approximate the sets $U_n$ from inside by unions of cylinders, and prove that the entry times distribution will converge to the same limit;
	\item finally, we verify that the return maps for the Young's towers, $T=f^R$, satisfy the assumptions above; then an inducing argument will carry the convergence to the original map $f$.
\end{enumerate} 
One thing to keep in mind is that, in this section, we will not use Assumption~\ref{a.3} or~\ref{a.4}. In the meantime, Assumption~\ref{a.1} and~\ref{a.2} are only used when one considers open sets $U_n$ (Theorem~\ref{t.4}).
\subsection{Compound binomial distribution of cylinder sets for $\phi$-mixing measures}
In this subsection, we let $T$ be a map on a probability space $\Omega$ and $\mu$ be a $T$-invariant probability measure 
on $\Omega$. We assume that there is a measurable partition (finite or countably infinite) $\mathcal{A}$ of 
$\Omega$ and denote by $\mathcal{A}^n=\bigvee_{j=0}^{n-1}T^{-j}\mathcal{A}$  its $n$th join. $\mathcal{A}^n$ is a partition of $\Omega$ and its elements are
called $n$-cylinders. For a point $x\in\Omega$ we denote by $A_n(x)\in\mathcal{A}^n$ 
the unique $n$-cylinder that contains the point $x$. 
We assume that $\mathcal{A}$ is generating, that is $\bigcap_nA_n(x)$ consists of the 
singleton $\{x\}$.

\begin{definition}  The measure $\mu$ is {\em left $\phi$-mixing} with respect to  $\mathcal{A}$
	if
	$$
	|\mu(A \cap T^{-n-k} B) - \mu(A)\mu(B)| \le \phi(k)\mu(A)
	$$
	for all $A \in \sigma(\mathcal{A}^n)$, $n\in\mathbb{N}$ and  $B \in \sigma(\cup_j \mathcal{A}^j )$,
	where $\phi(k)$ is a decreasing function which converges to zero as $k\to\infty$. Here $\sigma(\mathcal{A}^n)$ is the $\sigma$-algebra generated by $n$-cylinders.\\
\end{definition}

For simplicity we will drop the superscript in $\zeta$ (as it is always coupled with the measure of $U$) and write 
$$
\zeta_U = \zeta^{\tau/\mu(U)}_{U} =  \sum_{k=0}^{\tau/\mu(U)-1}\II_{U}\circ f^k.
$$
We will also write $S\lesssim B$ if there is a universal constant $C$ such that $S\le C\cdot B$.  Recall that $\lambda(K,U)$ is defined by~\eqref{e.2} and $\alpha_1(K,U) = \mu_{U}(\tau_{U}>K).$

The next theorem is the compound binomial approximation for a union of $n$-cylinders. A similar result is obtained in~\cite{GHV} using the Chen-Stein method. Here we will prove it using the compound binomial distribution theorem in Section~\ref{s.3}.

To simplify notation, we let $K$ be an integer and put $Z_j = \sum_{i=j(2K+1)}^{(j+1)(2K+1)-1}X_i$ as the $j$th block, with $X_i=\II_{U}\circ T^{i}$  as before. We will also write
$$
\phi^1(k)=\sum_{j=k}^{\infty}\phi(j)
$$
for the tail sum of $\phi$.

\begin{theorem}\label{t.3}
Let $\mu$ be a $T$-invariant probability measure that is left $\phi$-mixing with respect to an at most countable, generating partition $\cA$. Assume that $\phi(k)$ is summable in $k$. Let $U\in\sigma(\cA^n)$ be a union of $n$-cylinders with positive measure.

Then there exists a constant $C>0$, such that for all integers $K,\Delta$ with  $\Delta(2K+1) < \tau/\mu(U)$ and every $k\in\NN_0$, one has
\begin{align*}\numberthis\label{e.5.1}
\left|\PP(\zeta_U = k)- m(\{k\})\right|\le& C\phi(\Delta/2)+(2K+2)\Delta\mu(U)+\phi^1(K)\\&+\frac{\tau}{(2K+1)\mu(U)}\sum_{j=1}^{j_0}\PP\left(Z_0\ge 1 \,\wedge\,Z_j\ge 1\right),
\end{align*}
where $m$ is compound binomial with parameters $(\tau/((2K+1)\mu(U)),\PP(\tau_U\le 2K))$ on the binomial part,   $\{\lambda_\ell(K,U)\}$ on the compound part, and $j_0=[n/(2K+1)]+2.$
\end{theorem}

\begin{proof}

We employ the compound binomial approximation theorem in Section~\ref{s.3} on $N=[\tau/\mu(U)]$. Put $V_a^b=\sum_{j=a}^b Z_j$. Then for any $2\le\Delta\le N'=N/(2K+1)$ (where we assume $N'$ is an integer for simplicity), we have
$$
\left|\PP(V_0^{N'}=k)- m(\{k\})\right|\le CN'(\cR_1+\cR_2)+\Delta\mu(U),
$$
where
\begin{flalign*}
\cR_1=\sup\limits_{\substack{M\in[\Delta,N']\\q\in(0,N'-\Delta)}}\Bigg|\sum_{u=1}^{q-1}\big(\PP(Z_0=u\,\wedge\,&V_\Delta^M=q-u)-\PP(Z_0=u)\PP\left(V_\Delta^M=q-u\right)\big)\Bigg|,&
\end{flalign*}
and
\begin{flalign*}
\cR_2=\sum_{j=1}^{\Delta}\PP&\left(Z_0\ge 1 \,\wedge\,Z_j\ge 1\right).&
\end{flalign*}
Here $m$ is the compound binomial distribution with parameter $N'=\tau/((2K+1)\mu(U))$, $p=\PP(Z_j\ge 1) = \PP(\tau_U\le 2K)$ in the binomial part, and $(1/p)\PP(Z_j=\ell)=\lambda_\ell(K,U)$ in the compound part.

Next we will estimate the error terms $\cR_1$ and $\cR_2$ using the left $\phi$-mixing property. We will also use the following trivial estimate:
$$
\PP(Z_0\ge 1) =\mu(\bigcup_{i=0}^{2K}T^{-i}U)\le (2K+1)\mu(U).
$$

\noindent 1. Estimate $\cR_1$.

Note that $\{Z_0=u\}\in\sigma(\cA^{n+2K+1})$, and $\{V_\Delta^M=q-u\}\in T^{-\Delta}\sigma(\bigcup_j\cA^j)$. Therefore, if $\Delta\ge 2(n+2K+1)$ then we get from the mixing property,
\begin{align*}
\cR_1\le&\, \phi(\Delta/2)\sum\limits_{u=1}^{q-1}\PP(Z_0=u)\\
\le&\,\phi(\Delta/2)\PP(\tau_U\le 2K+1)\\
\le&\,\phi(\Delta/2)(2K+1)\mu(U).
\end{align*}

\noindent 2. Estimate $\cR_2$.
To estimate $\cR_2$, we first note that since $\{Z_0\ge 1\}\in\sigma(\cA^{n+2K+1})$, and $n$ will be sent to infinity while $K$ is fixed~\footnote{This fact is not used in this theorem, but is essential for the convergence to the compound Poisson distribution in our setup, as the convergence of parameters $\lambda_\ell(K,U)$ requires two separate limits.}, one cannot use the mixing assumption on $\PP\left(Z_0\ge 1 \,\wedge\,Z_j\ge 1\right)$ for small values of $j$. 

To solve this issue, we write
$$
j_0=[n/(2K+1)]+2.
$$
When  $j\ge j_0$, we have a gap between $\{Z_0\ge 1\}\in\sigma(\cA^{n+2K+1})$ and $\{Z_j \ge 1\}\in T^{-j(2K+1)}\sigma(\bigcup_j\cA^j)$ with size at least $K$. The mixing property then yields:
\begin{align*}
\PP\left(Z_0\ge 1 \,\wedge\,Z_j\ge 1\right)\le &\PP(Z_0\ge 1)(\PP(Z_j\ge 1)+\phi((j-1)(2K+1)-n))\\
=&\PP(Z_0\ge 1)^2+ \PP(Z_0\ge 1)\phi((j-1)(2K+1)-n),
\end{align*}
where the second line follows from stationarity. Sum over $j>j_0$ and recall that $\phi(k)\lesssim k^{-p}$ for some $p>1$, we obtain
\begin{align*}
&\sum_{j=j_0}^{\Delta}\PP\left(Z_0\ge 1 \,\wedge\,Z_j\ge 1\right)\\
\le &\sum_{j=j_0}^{\Delta}\PP(Z_0\ge 1)^2+ \PP(Z_0\ge 1)\phi((j-1)(2K+1)-n)\\
\le &\Delta(2K+1)^2\mu(U)^2+(2K+1)\mu(U)\sum_{j\ge j_0}\phi((j-1)(2K+1)-n)\\
\lesssim &\Delta(2K+1)^2\mu(U)^2+(2K+1)\mu(U)\phi^1(K).
\end{align*}

\noindent 3. Collect the error terms.

Now we collect the estimates above and obtain (recall that $N'=N/(2K+1) = \tau/((2K+1)\mu(U))$):
\begin{align*}
&\left|\mu(\zeta_n = k)- m(\{k\})\right|\\
\lesssim& N'\big(\phi(\Delta/2)(2K+1)\mu(U)+\Delta(2K+1)^2\mu(U)^2\\&+\sum_{j=1}^{j_0}\PP\left(Z_0\ge 1 \,\wedge\,Z_j\ge 1\right)+(2K+1)\mu(U)\phi^1(K)\big) +\Delta\mu(U)\\
\lesssim& \phi(\Delta/2)+(2K+2)\Delta\mu(U)+\phi^1(K)+\frac{\tau}{(2K+1)\mu(U)}\sum_{j=1}^{j_0}\PP\left(Z_0\ge 1 \,\wedge\,Z_j\ge 1\right).
\end{align*}

\end{proof}
\begin{remark}\label{r.5.2}
The first two terms on the right-hand-side of~\eqref{e.5.1} will converge to zero if one considers a sequence of nest sets $U_n$ with $\mu(U_n)\to 0$ and let $\Delta = \mu(U_n)^{-1/2}$. The third term can be dealt with by sending $K$ to infinity on a second limit (recall that $\phi$ is assumed to be summable). Doing so will also make the compound binomial distribution $m$ converge to a compound Poisson distribution with the desired parameters $\{\tau\alpha_1\lambda_\ell\}$, as we have seen in Remark~\ref{r.2}, Section~\ref{ss.3.2}. However, controlling the last term will require more information on other structures of the system. This is carried out in the next subsection.
\end{remark}
\subsection{Gibbs Markov systems}

Recall that a map $T:\Omega\to\Omega$ is called {\em Markov} if there is a countable measurable partition $\cA$ on $\Omega$ with $\mu(A)>0$ for all $A\in \cA$, such that for all $A\in \cA$, $T(A)$ is injective and can be written as a union of elements in $\cA$. Write $\cA^n=\bigvee_{j=0}^{n-1}T^{-j}\cA$ as before, it is also assumed that $\cA$ is (one-sided) generating.

Fix any $\lambda\in(0,1)$ and define the metric $d_\lambda$ on $\Omega$ by $d_\lambda(x,y) = \lambda^{s(x,y)}$, where $s(x,y)$ is the largest positive integer $n$ such that $x,y$ lie in the same $n$-cylinder. Define the Jacobian $g=JT^{-1}=\frac{d\mu}{d\mu\circ T}$ and $g_k = g\cdot g\circ T \cdots g\circ T^{k-1}$.

The map $T$ is called {\em Gibbs-Markov} if it preserves the measure $\mu$, and also satisfies the following two assumptions:\\
(i) The big image property: there exists $C>0$ such that $\mu(T(A))>C$ for all $A\in \cA$.\\
(ii) Distortion: $\log g|_A$ is Lipschitz for all $A\in\cA$.

 For example, if a differentiable map $f$ is modeled by Young's towers with a base $\Omega_0$, then the return map $T = f^R:\Omega_0\to\Omega_0$ is a Gibbs-Markov map with respect to the invariant measure $\mu|_{\Omega_0}=(h\nu)|_{\Omega_0}$ and the partition $\{\Omega_{0,i}\}$, since $T(\Omega_{0,i}) =\Omega_0$.

 In view of (i) and (ii), there exists a constant $D>1$ such that for all $x,y$ in the same $n$-cylinder, we have the following distortion bound:
$$
\left|\frac{g_n(x)}{g_n(y)}-1\right|\le D d_\lambda(T^nx,T^ny),
$$
and the Gibbs property:
$$
D^{-1}\le \frac{\mu(A_n(x))}{g_n(x)}\le D.
$$
It is well known (see, for example, Lemma 2.4(b) in~\cite{MN05}) that Gibbs-Markov systems are exponentially $\phi$-mixing, that is, $\phi(k)\lesssim \eta^k$ for some $\eta\in(0,1)$.

Before stating the next theorem, we will make some assumption on the sizes of the nested sequence $\{U_n\}$. We assume that each $U_n$ is a union of $\kappa_n$-cylinders, for some integers $\kappa_n\to\infty$ as $n\to\infty$. For each $n$ and $j\ge 1$, we define $\cC_j(U_n) =\{A\in\cA^j,C\cap U_n\ne\emptyset\}$ the collection of all $j$-cylinders that have non-empty intersection with $U_n$.
Then we write 
$$
U_n^j = \bigcup_{A\in\cC_j(U_n)} A
$$
for the approximation of $U_n$ by $j$-cylinders from outside. For each fixed $j$, $\{U_n^j\}_n$ is also nested, that is, $U_{n+1}^j\subset U_n^j$. Obviously we have $U_n\subset U_n^j$ for all $j$, and $U_n=U_n^j$ if $j\ge \kappa_n$. Also note that the diameter of $j$-cylinders are exponentially small in $j$. Together with the distortion property (ii), we see that the measure of $j$-cylinders are also exponentially small in $j$.

The next theorem shows the convergence to the compound Poisson distribution for a nest sequence of cylinder sets $U_n$, which is interesting in its own right:

\begin{main}\label{m.4}
Let $T$ be a Gibbs-Markov systems and $U_n\in\sigma(\cA^{\kappa_n})$ a sequence of nested sets with $\kappa_n\mu(U_n)\to 0$. Assume that $\{\hat{\alpha}_{\ell}\}$ defined in~\eqref{e.3} exists, and satisfies $\sum_{\ell}\ell\hat\alpha_{\ell}<\infty$. We also assume that there are constants $C>0$ and $p'>1$ such that $\mu(U_n^j)\le \mu(U_n) + Cj^{-p'}$ for every $j\le \kappa_n$ .

Then the entry times distribution $\zeta_{U_n}= \sum_{k=0}^{\tau/\mu(U_n)-1}\II_{U_n}\circ f^k$ satisfy
$$
\PP(\zeta_{U_n} = k)\to m(\{k\})
$$
as $n\to\infty$ for every $k\in\NN_0$, where $m$ is the compound Poisson distribution with parameters $\{\tau\alpha_{1}\lambda_\ell\}$ with $\lambda_\ell$, $\alpha_1$ defined by~\eqref{e.3'} and~\eqref{e.4} respectively.
\end{main}

\begin{proof}
In view of Theorem~\ref{t.3}, Remark~\ref{r.2} and~\ref{r.5.2}, we only need to show that the last term in~\eqref{e.5.1}:
$$
\frac{\tau}{(2K+1)\mu(U_n)}\sum_{j=1}^{j_0}\PP\left(Z_0\ge 1 \,\wedge\,Z_j\ge 1\right)
$$
converges to zero under the limit in $n$ followed by a limit in $K$. Here $j_0 =\kappa_n/(2K+1)]+2$ and $Z_j =\sum_{i=j(2K+1)}^{(j+1)(2K+1)-1}X_i$.

We start with some observations on the Gibbs-Markov systems. First, By the big image property, for any $j$-cylinder $A\in\cA^j$, we have
\begin{equation}\label{e.5.2}
\mu(T^jA)\ge C,
\end{equation}
where $C$ is the constant from (i).

Secondly, for any $j$-cylinder $A$ and any set $U\in \Omega$, the distortion property (ii) and the Gibbs property gives
\begin{equation}\label{e.5.3}
\frac{\mu(U\cap A)}{\mu(A)}\le D\frac{\mu(T^j(U\cap A))}{\mu(T^jA)}.
\end{equation}

Now we are ready to estimate  $\mu(U_n\cap T^{-j}U_n)$:
\begin{align*}
\mu(U_n\cap T^{-j}U_n)\le& \sum_{A\in\cC_j(U_n)}\mu(T^{-j}U_n\cap A)\\
=&\sum_{A\in\cC_j(U_n)}\frac{\mu(T^{-j}U_n\cap A)}{\mu(A)}\mu(A)\\
\lesssim&\sum_{A\in\cC_j(U_n)}\frac{\mu(T^j(T^{-j}U_n\cap A))}{\mu(T^jA)}\mu(A)\\
\lesssim&\sum_{A\in\cC_j(U_n)}\mu(U_n)\mu(A)\\
=&\mu(U_n)\mu\left(\bigcup_{A\in\cC_j(U_n)}A\right)=\mu(U_n)\mu(U_n^j),
\end{align*}
where we use~\eqref{e.5.3} and~\eqref{e.5.2} on the third and forth line, respectively.

Then for $j\ge 2$, 
\begin{align*}
\PP(Z_0\ge 1, Z_j\ge 1) \le& \sum_{0\le k,\ell<2K+1}\mu(T^{-k}U_n\cap T^{-\ell-j(2K+1)}U_n)\\
=&\sum\limits_{u=(j-1)(2K+1)}^{(j+1)(2K+1)}((2K+1)-|u-j(2K+1)|)\mu(U_n\cap T^{-u}U_n)\\
\le&(2K+1)\sum\limits_{u=(j-1)(2K+1)}^{(j+1)(2K+1)}\mu(U_n\cap T^{-u}U_n).
\end{align*}
Summing over $j$ from  $2$ to $j_0$, we get
\begin{align*}
\sum_{j= 2}^{j_0}\PP(Z_0\ge 1, Z_j\ge 1)\le &2(2K+1)\sum\limits_{u=(2K+1)}^{(j_0+1)(2K+1)}\mu(U_n\cap T^{-u}U_n)\\
\lesssim& (2K+1)\mu(U_n)\sum_{u=2K+1}^{\kappa_n+2(2K+1)} \mu(U_n^u)\\
\lesssim& (2K+1)\mu(U_n)\left((\kappa_n+4(2K+1))\mu(U_n)+\sum_{u= 2K+1}^{\kappa_n} u^{-p'}\right),
\end{align*}
where the last line follows from the assumption that $\mu(U_n^j)\le \mu(U_n) + Cj^{-p'}$ and the observation that $U_n^j=U_n$ for $j\ge \kappa_n$. Divide by $(2K+1)\mu(U_n)$, we see that 
\begin{align*}
&\frac{\tau}{(2K+1)\mu(U_n)}\sum_{j=2}^{j_0}\PP\left(Z_0\ge 1 \,\wedge\,Z_j\ge 1\right)\\
\lesssim&(\kappa_n+4(2K+1))\mu(U_n)+\sum_{u\ge 2K+1} u^{-p'}\\\lesssim&(\kappa_n+4(2K+1))\mu(U_n)+ K^{-(p'-1)}.
\end{align*}
The first term goes to zero with $n\to\infty$, and the second term vanishes with $K\to\infty$ (recall that $p'>1$).

We are only left with $\PP(Z_0\ge1, Z_1\ge 1)$. We take any $K'<K$ and split the sum in $Z_0$ as:
$$
Z_0'=\sum_{i=2K+1-K'}^{2K+1}X_i,\mbox{ and }\, Z_0''=Z_0-Z_0'.
$$
Then
\begin{align*}
\PP(Z_0\ge 1,Z_1\ge 1)\le& \PP(Z_0''\ge 1,Z_1\ge 1)+\PP(Z_0'\ge 1)\\
\le& \PP(Z_0''\ge 1,Z_1\ge 1)+K'\mu(U_n).
\end{align*}
For the first term on the right-hand-side, we follow the previous estimate to obtain:
\begin{align*}
\PP(Z_0''\ge 1, Z_1\ge 1) \le& \sum_{\substack{0\le k\le 2K+1-K'\\0\le\ell<2K+1}}\mu(T^{-k}U_n\cap T^{-\ell-(2K+1)}U_n)\\
\le&(2K+1)\sum\limits_{u=K'}^{2(2K+1)}\mu(U_n\cap T^{-u}U_n)\\
\le&(2K+1)\mu(U_n)\sum_{u=K'}^{2(2K+1)}\mu(U_n^u)\\
\lesssim&(2K+1)\mu(U_n)\left(2(2K+1)\mu(U_n) + \sum_{u=K'}^{2(2K+1)}u^{-p'}\right).
\end{align*}
Divide by $(2K+1)\mu(U_n)$, we obtain that for any $K'<K$,
\begin{align*}
\frac{\tau}{(2K+1)\mu(U_n)}\PP(Z_0\ge 1,Z_1\ge 1)\lesssim K\mu(U_n)+(K')^{-(p'-1)}+\frac{K'}{K}.
\end{align*}
If we choose $K'=\sqrt{K}$ then all three terms converge to zero under the limit $n\to\infty$ then limit in $K\to\infty$. This finishes the proof of Theorem~\ref{m.4}.
\end{proof}

\begin{remark}
The assumption that $\kappa_n\mu(U_n)\to 0$ is very mild, as the measure of $\kappa_n$ cylinders are of the order $\lambda ^{-\kappa_n}$, so one allows the number of $\kappa_n$-cylinders in $U_n$ to be exponentially large in $\kappa_n.$

Same can be said about the assumption $\mu(U_n^j)\le \mu(U_n) + Cj^{-p'}$. In fact, we will see in the next subsection that the difference between $U_n^j$ and $U_n$ are precisely those $j$-cylinders that cross the topological boundary of $U_n$. 
\end{remark}

Recall that $\beta_\ell$ is defined by~\eqref{e.beta} as an alternative way to study the short return properties of $\{U_n\}$ by synchronizing $K$ and $n$ (thus taking only one limit). As a by-product of the previous theorem, we have the following relation between $\{\beta_\ell\}$ and $\{\hat\alpha_\ell\}$:

\begin{proposition}\label{p.1}
Under the assumptions of Theorem~\ref{m.4}, for any increasing sequence  $\{s_n\}$ with $s_n\to\infty$ and $s_n\mu(U_n)\to 0$, the sequence $\{\beta_\ell\}$ defined by~\eqref{e.beta} exists and satisfies
$
\beta_\ell=\hat{\alpha}_\ell
$
for all $\ell \ge 1$.
\end{proposition}

\begin{proof}
We estimate $|\beta_\ell-\hat{\alpha}_\ell|$ by writing:
\begin{align*}
&|\mu_{U_n}(\tau_{U_n}^{\ell-1}\le s_n) - \mu_{U_n}(\tau^{\ell-1}_{U_n}\le 2K)|\\
\le& \PP_{U_n}\left(\sum_{i=2K+1}^{s_n}X_i\ge1\right)\\
\le& \frac{1}{\mu(U_n)}\sum_{j=1}^{s_n/(2K+1)}\PP(X_0=1, Z_j\ge 1)\\
\le& I+ II,
\end{align*}
where $I$ is the sum over $j$ from $1$ to $j_0=[\kappa_n/(2K+1)]+2$, and $II$ is the sum from $j_0$ to $s_n/(2K+1)$.

For $II$, we follow the estimation of $\cR_2$ in the proof of Theorem~\ref{t.3} and obtain by the $\phi$-mixing assumption: 
\begin{align*}
II\le& \frac{1}{\mu(U_n)}\sum_{j=j_0}^{s_n/(2K+1)} \mu(U_n)\PP(Z_0\ge 1)+\mu(U_n)\phi((j-1)(2K+1)-\kappa_n)\\
\le& \frac{s_n\mu(U_n)}{2K+1}+\phi^1(K),
\end{align*}
where $\phi^1$ is the tail sum of $\phi$ as before. 

For $I$, we use the argument in the proof of Theorem~\ref{m.4} and get:
\begin{align*}
I=&\frac{1}{\mu(U_n)}\sum_{j=1}^{j_0}\PP(X_0=1,Z_j\ge 1) \\
\le&\frac{1}{\mu(U_n)}\sum_{j=1}^{j_0} \sum_{k=j(2K+1)}^{(j+1)(2K+1)}\mu(U\cap T^{-k}U_n)\\
=&\frac{1}{\mu(U_n)}\sum_{k=2K+1}^{(j_0+1)(2K+1)}\mu(U\cap T^{-k}U_n)\\
\le&\frac{1}{\mu(U_n)}\sum_{k=2K+1}^{(j_0+1)(2K+1)}\mu(U_n)\mu(U^j_n)\\
\le& (\kappa_n+4(2K+1))\mu(U_n)+\sum_{j=2K+1}^{\kappa_n}j^{-p'}\\
\lesssim&(\kappa_n+4(2K+1))\mu(U_n)+K^{-(p'-1)}.
\end{align*}

Collect the estimations above and send $n$ to infinity, we see that
$$
|\beta_\ell-\hat{\alpha}_{\ell}(2K)|\lesssim\phi^1(K)+ K^{-(p'-1)},
$$
which converges to zero following the limit in $K$. This concludes the proof.
\end{proof}
Recall that $\alpha_1$ is defined by~\eqref{e.4} and satisfies $\alpha_1=\hat{\alpha}_1-\hat\alpha_2$. 
In particular, this proposition shows that $\alpha_1$ coincides with the extremal index $\theta=\lim_n \mu_{U_n}(\tau_{U_n}> s_n)$ defined in~\cite{FFRS}.

\subsection{From cylinders to open sets}
Now we shift our attention to $U_n$'s that are not necessarily unions of cylinder sets. For this purpose, let $\Omega = \bM$ be a compact Riemannian manifold and $T=f$ a differentiable map on $\bM$. We will still assume that there is a (at most) countable partition $\cA$, with respect to which the system is Gibbs-Markov. 
Examples of such systems include Markov interval maps, higher dimensional expanding maps with Markov partition, and the return map of Young's towers for non-invertible maps. 

We will take $\{U_n\}$ a sequence of nested open sets with measure converging to zero. In particular, one can take a continuous function $\vp:\bM\to\RR\cup\{\pm\infty\}$ and a sequence of thresholds $\{u_n\}$, and let $U_n = \{x:f(x)>u_n\}$. As before, we are interested in the limiting distribution of $\PP(\zeta_{U_n}=k)$, which can be immediately translated into the distribution of $\PP(\xi^{w_n}_{u_n}=k)$ according to Theorem~\ref{m.3}, where $\{w_n\}$ is a sequence of integers satisfying~\eqref{e.1}.

Note that the sets $U_n$ are very  likely not unions of cylinders in $\cA^n$, thus one cannot directly apply Theorem~\ref{m.4}. To solve this issue, we will approximate $U_n$ by unions of cylinders from inside. Given any set $U\subset M$ and $\rho>0$, we write  $B_\rho(U) = \bigcup_{x\in U}B_\rho(x)$ for the $\rho$-neighborhood of $U$.

The main result of this section is the following theorem:

\begin{theorem}\label{t.4}
Let $(\bM,\mu,f,\cA)$ be a Gibbs-Markov system, and $\{U_n\}$ be a nested sequence of open sets that satisfies Assumption~\ref{a.1}, with $\{\hat{\alpha}_{\ell}\}$ exists and satisfies $\sum_{\ell}\ell\hat{\alpha}_{\ell}<\infty$. Write $\kappa_n$ the smallest positive integer with $\diam\cA^{\kappa_n}\le r_n$ where $r_n$ is the sequence in Assumption~\ref{a.1}.  We assume that:
\begin{enumerate}[label=(\alph*)]
\item $\kappa_n\mu(U_n)\to 0$;
\item $U_n$ have small boundary: there exists $C>0$ and $p'>1$, such that \\$\mu\left(\bigcup_{A\in\cA^j, A\subset  B_{r_n}(\partial U_n) }A\right)\le C j^{-p'}$ for all $j\le\kappa_n$.
\end{enumerate}

Then the entry times distribution $\zeta_{U_n}= \sum_{k=0}^{\tau/\mu(U_n)-1}\II_{U_n}\circ f^k$ satisfy
$$
\PP(\zeta_{U_n} = k)\to m(\{k\})
$$
as $n\to\infty$ for every $k\in\NN_0$, where $m$ is the compound Poisson distribution with parameters $\{\tau\alpha_{1}\lambda_\ell\}$.

In particular, the rare event process $\xi$ has the same limiting distribution:
$$
\PP(\xi^{w_n}_{u_n} = k)\to m(\{k\}),
$$
where $w_n$ is a sequence of integers given by~\eqref{e.1}.
\end{theorem}

To prove this theorem we first introduce some notations. Let $r_n$ be the sequence of real numbers given by Assumption~\ref{a.1}. For each $n$, we take $\kappa_n$ to be the smallest integer such that $\diam(\cA^{\kappa_n})\le r_n$, and put
$$
V_n = \bigcup_{A\in\cA^{\kappa_n},A\subset U_n}A.
$$
In other words, $V_n$ is the approximation of $U_n$ from inside by $\kappa_n$-cylinders. Due to the choice of $\kappa_n$, we have 
$$
U_n^i\subset V_n\subset U_n,
$$
and
\begin{equation}\label{e.5.4}
\frac{\mu(U_n\setminus V_n)}{\mu(U_n)} = o(1).
\end{equation}
It then follows that $\kappa_n\mu(V_n)\to 0$, provided that $\kappa_n\mu(U_n)\to 0.$

Theorem~\ref{m.4} requires us to estimate the measure of $V_n^j$, which is the union of $j$-cylinders that has non-empty intersection with $V_n$. Observe that
$$
V_n^j\subset \bigcup_{A\in\cA^j, A\cap U_n\ne\emptyset }A;
$$
moreover, the difference between $V_n$ and $V_n^j$ satisfies
$$
V_n^j\setminus V_n\subset \bigcup_{A\in\cA^j, A\subset  B_{r_n}(\partial U_n) }A.
$$
This gives
\begin{equation}
\mu(V_n^j)\le \mu(V_n)+\mu\left(\bigcup_{A\in\cA^j, A\subset  B_{r_n}(\partial U_n) }A\right)\lesssim \mu(V_n)+j^{-p'},
\end{equation}
thanks to the assumption (b).

Then we can apply Theorem~\ref{m.4} on the sequence of nested cylinder sets $V_n\in\cA^{\kappa_n}$ to get
$$
\PP(\zeta_{V_n} = k)\to m(\{k\})
$$
where $m$ is the compound Poisson distribution with parameters $\{\alpha^V_{1}\tau\lambda^V_\ell\}$ defined using $\{V_n\}$. It remains to show that the parameters $\alpha^U_{1}$, $\lambda^U_\ell$ defined using $\{U_n\}$ coincides with those defined using $\{V_n\}$, and that $\zeta_{U_n}$ has the same limiting distribution with $\zeta_{V_n}$.

In view of Theorem~\ref{t.1}, to prove that parameters satisfy $\lambda^U_\ell=\lambda^V_\ell$, we only need to show the following lemma:
\begin{lemma}\label{l.5.5}
Let $V_n, U_n$ be two sequences of nested sets with $V_n\subset U_n$ for each $n$. Put 
$$
\hat\alpha^*_\ell = \lim_{K\to\infty}\lim_{n\to\infty}\mu_{*_n}(\tau^{\ell-1}_{*_n}\le K),*=U,V.
$$
Then $\hat\alpha^U_\ell=\hat\alpha^V_\ell$ provided that~\eqref{e.5.4} holds.
\end{lemma}
\begin{proof}
To simplify notation, we drop the index on $U_n$. We have to estimate:
\begin{align*}
&\left|\mu_{U}(\tau^{\ell-1}_{U}\le K)-\mu_{V}(\tau^{\ell-1}_{V}\le K)\right|\\
\le &\frac{1}{\mu(U)}\left|\mu(\tau^{\ell-1}_{U}\le K\wedge U)-\mu(\tau^{\ell-1}_{V}\le K\wedge V)\right|+\frac{\mu(U\setminus V)}{\mu(U)}\mu_{V}(\tau^{\ell-1}_{V}\le K)\\
\le &\frac{1}{\mu(U)}\left|\mu(\tau^{\ell-1}_{U}\le K\wedge U)-\mu(\tau^{\ell-1}_{V}\le K\wedge U)\right|
+\frac{1}{\mu(U)}\mu(U\setminus V)\\
&+\frac{\mu(U\setminus V)}{\mu(U)}\mu_{V}(\tau^{\ell-1}_{V}\le K).
\end{align*}
The second and third term on the right-hand-side converge to zero as $n\to \infty$, thanks to~\eqref{e.5.4}. The first term is estimated as
\begin{align*}
&\frac{1}{\mu(U)}\left|\mu(\tau^{\ell-1}_{U}\le K\wedge U)-\mu(\tau^{\ell-1}_{V}\le K\wedge U)\right|\\
\le&\mu_U(\tau_{U\setminus V}\le K)\\
\le& \frac{1}{\mu(U)} K\mu(U\setminus V)\to 0\mbox{ as }n\to\infty.
\end{align*}
This finishes the proof of the lemma.
\end{proof}
As a simple consequence of this lemma, we have $\sum_{\ell}\ell\hat{\alpha}^V_{\ell}<\infty$, provided that $\sum_{\ell}\ell\hat{\alpha}^U_{\ell}<\infty$.

Finally we control the difference between $\zeta_{U_n}$ and $\zeta_{V_n}.$
\begin{lemma}\label{l.5.6}
Assume that $\{U_n\}$, $\{V_n\}$ are two sequences of nested sets with $V_n\subset U_n$. Moreover, assume that~\eqref{e.5.4} holds. Then 
$$ 
|\PP(\zeta_{U_n}=k)-\PP(\zeta_{V_n}=k)|\to 0
$$
as $n\to\infty$.
\end{lemma}
\begin{proof}
First note that $\frac{1}{\mu(V_n)}>\frac{1}{\mu(U_n)}$, i.e., $\zeta_{V_n}$ contains more terms. We have
\begin{align*}
|\PP(\zeta_{U_n}=k)-\PP(\zeta_{V_n}=k)|\le& \PP\left(\tau_{U_n\setminus V_n}\le \frac{1}{\mu(U_n)}\right)+\PP\left(\sum\limits_{i=1/\mu(U_n)}^{1/\mu(V_n)}\II_{U_n}\circ f^i\ge 1\right)\\
\le&\frac{1}{\mu(U_n)}\mu(U_n\setminus V_n) + \mu(U_n)\left(\frac{1}{\mu(V_n)}-\frac{1}{\mu(U_n)}\right)\to 0.
\end{align*}
The proof is finished.
\end{proof}
With these lemmas, we conclude the proof of Theorem~\ref{t.4}.
\begin{remark}
It can be seen from the proof that one does not need $V_n$ to be subsets of $U_n$. If $\{V_n\}$ is a nested sequence such that
$$
\frac{\mu(U_n\dotminus V_n)}{\mu(U_n)}\to 0 \mbox{ as } n\to\infty
$$
where $\dotminus$ is the symmetric difference, then the same proof will show that $\hat{\alpha}^U_\ell = \hat{\alpha}^V_\ell$ for all $\ell\ge 1$. Similarly, $\zeta_{U_n}$ and $\zeta_{V_n}$ must converge to the same compound Poisson distribution with the same parameters.

Note that the proof of the previous two lemmas does not require the system to be Gibbs-Markov or even mixing. The proof also applies to invertible systems without any change.
\end{remark}
\begin{remark}
Note that in the assumption (b), we only need the estimate on the boundary of $U_n$ for $j<\kappa_n$. This coincides with our statement of Assumption~\ref{a.2} at the beginning of the paper, where we need $r$ to be small but not too small, where the lower bound depends on $n$.
\end{remark}

Assumption (a) of Theorem~\ref{m.4} is rather mild. Normally the sets $U_n$ are the $\rho_n$-neighborhood of $\Lambda$ for some $\rho_n>0$, and the measure of $U_n$ are of order $\rho_n^a$ for some $a>0$. Then Assumption~\ref{a.1} holds with $r_n=\rho_n^b$ for some $b>1$ large enough. Since the diameter of $n$-cylinders are exponentially small, $\kappa_n$ is of order $|\log \rho_n|$. In this case, $\kappa_n\mu(U_n)\to 0$ holds.

On the other hand, to achieve assumption (b) in Theorem~\ref{t.4}, note that $\diam \cA^j\lesssim \lambda^j$, so $V_n^j$ is the $(\rho_n+\lambda^j)$-neighborhood of $\Lambda$, and  $V^j_n\setminus V_n$ consists of the $(\lambda^j)$-neighborhood of $\partial U_n$, whose measure can be controlled if $\mu$ is absolutely continuous with respect to the volume on $\bM$ and if $\Lambda$ is `nice' (for example, a embedded submanifold with dimension less than $\dim\bM$).

We conclude this subsection with the following observation:

\begin{remark}\label{r.invertible}
Note that the proof of the previous theorems does not depend on whether the system is non-invertible or not. In particular, Theorem~\ref{t.4} holds when the systems is invertible and $\phi$-mixing (where the partition $\cA^n$ should be defined using two-sided joint, i.e., $\cA^n = \bigvee_{i=-n}^n f^{-i}\cA$). Such systems include Axiom A diffeomorphisms with equilibrium states and dispersing billiards.

Similarly, Theorem~\ref{m.4} also holds for the systems mentioned above, as the only ingredient in the proof is the distortion estimate, which holds as long as the systems has sufficient hyperbolicity. 
\end{remark}

\subsection{The inducing argument}
Now let $f$ be a non-invertible, differentiable map $f$ on a compact manifold $\bM$, preserving an invariant measure $\mu$. We assume that $\vp:\bM\to\RR\cup\{\pm\infty\}$ is a continuous function that achieves its maximum on a set $\Lambda$ with zero measure. The following theorem is proven in~\cite{FFM}:
\begin{theorem}\cite[Theorem 2.C]{FFM}\label{t.5}
Assume that there is a set $\Omega\subset \bM$ with positive $\mu$ measure, with $\Lambda\subset \Omega$. Assume that there is a sequence of thresholds $\{u_n\}$, such that the sets $U_n=\{x:\vp(x)>u_n\}$ are contained in $\Omega$ for $n$ large enough. Moreover, assume that the induced map $T:\Omega\to\Omega$ is defined using the first return map of $f$ on $\Omega$, such that the return time function is integrable with respect to the induced measure $\mu_0=\mu|_\Omega$. 

Then if the rare event process $\xi$ for the induced system $(\Omega,T,\mu_0)$ is compound Poisson distributed, so is the rare event process for the original system $(\bM,f,\mu)$.
\end{theorem}

Then if $f$ is modeled by Young's tower defined using the first return map, and $\U_n\subset \Omega_0$ (as we have assumed at the beginning of this section), then one can apply Theorem~\ref{t.4} on the induced Gibbs-Markov map $T=f^R$ to obtain the compound Poisson distribution for the induced rare event process. Theorem~\ref{t.5} will then guarantee that the original system has the same distribution. This finishes the proof of Theorem~\ref{m.2}, under the extra assumption that $U_n$ is contained in the induced base $\Omega_0$ and that the tower is defined using the first return map. As a trade-off, we  do not need (the very technical) Assumption~\ref{a.3}.

This result is not very satisfactory, however, as in higher dimensions, $\Omega_0$ is usually a Cantor set with empty interior, and the tower is often defined using a higher order return map\footnote{In the case of $C^{1+\alpha}$ surface diffeomorphisms, one can always construct towers using the first return map. See~\cite[Theorem B]{CLP}}. The general case of Theorem~\ref{m.2} will be dealt with in the next section.

\section{Proof of Theorem~\ref{m.1} and the general case of Theorem~\ref{m.2}}\label{s.5}
In this section, we will prove Theorem~\ref{m.2} and~\ref{m.1} using an argument that is similar to~\cite{HV19}, which was originally motivated by~\cite{CC}. Roughly speaking, we will approximate indicator functions $\II_{\{Z_0=u\}}$ by H\"older continuous functions $\phi$, and approximate $\II_{\{V_\Delta^M=q-u\}}$ by $L^\infty$ functions $\psi$ that are constant on stable disks. This will allow us to use decay of correlations~\eqref{e.decay2} to estimate terms like $\PP(Z_0=u,V_\Delta^M=q-u)$. More importantly, we do not need to consider the case $j\le j_0$ separately while controlling $\cR_2$. As a trade-off, we have to construct $\phi$ and $\psi$ very carefully, which will require assumptions on the topological boundary of $U_n$. Also note that the desynchronization between $n$ and $K$ plays an important role in the approximation. 

In view of Theorem~\ref{m.3}, we need to show that the hitting times distribution $\zeta_{U_n}$ converges to the compound Poisson distribution with parameters $\{\tau\alpha_{1}\lambda_\ell\}$. This is stated as the following theorem:

\begin{main}\label{m.5}
Let $f$ be either a $C^{1+\alpha}$ (local) diffeomorphism or a non-invertible map that can be modeled by Young's tower, with the decay rating satisfying $\cC(k) = o(1/k)$. Assume that $\{U_n\}$ is a sequence of nested sets with $\mu(U_n)\to 0$ and $\sum_\ell \ell\hat{\alpha}_\ell<\infty $. Furthermore, assume that Assumptions 1 to 4 hold with:
\begin{enumerate}
\item $r_n = o\left(\frac{\mu(U_n)}{\cC(\Delta_n/2)}\right)$ for a sequence $\Delta_n\nearrow\infty$ with $\Delta_n\mu(U_n)\to 0$; here $\cC$ is the rate in the decay of correlations given by~\eqref{e.decay1} or~\eqref{e.decay2}; 
\item for Assumption 2:
\begin{enumerate}
	\item in the non-invertible case,  $\mu(B_{r}(\partial U_n)) = \mathcal{O}(r^{p'})$ for $p'>1/\alpha$; here $\alpha>0$ is given by~\eqref{e.exp}; 
	\item in the invertible case, $\mu(B_{r}(\partial U_n)) = \mathcal{O}(r^{p'})$ for $p'>2/\alpha$;
\end{enumerate}
\item $p''>1$ in Assumption~\ref{a.3}.
\end{enumerate}
Then the entry times distribution $\zeta_{U_n}$ satisfies
$$
\PP(\zeta_{U_n} = k)\to m(\{k\})
$$
as $n\to\infty$ for every $k\in\NN_0$, where $m$ is the compound Poisson distribution with parameters $\{\tau\alpha_{1}\lambda_\ell\}$ with $\lambda_\ell$,  $\alpha_1$ defined by~\eqref{e.3'} and~\eqref{e.4} respectively.
\end{main}

The rest of this section is devoted to the proof of this theorem.

We use the same notation as in the last section. For an integer $K$ we write $Z_j=\sum_{i=j(2K+1)}^{(j+1)(2K+1)-1}X_i$, where $X_i=\II_{U_n}\circ f^{i}$. Then we apply Theorem~\ref{t.2} on the sequence $\{X_j\}$ with $N=\tau/\mu(U_n)$, and estimate 
$$
\left|\PP(V_0^{N'}=k)- m(\{k\})\right|\le CN'(\cR_1+\cR_2)+\Delta\mu(U),
$$
where $N'=\tau/((2K+1)\mu(U_n))$, and
\begin{flalign*}
\cR_1=\sup\limits_{\substack{M\in[\Delta,N']\\q\in(0,N'-\Delta)}}\Bigg|\sum_{u=1}^{q-1}\big(\PP(Z_0=u\,\wedge\,&V_\Delta^M=q-u)-\PP(Z_0=u)\PP\left(V_\Delta^M=q-u\right)\big)\Bigg|,&
\end{flalign*}
and
\begin{flalign*}
\cR_2=\sum_{j=1}^{\Delta}\PP&\left(Z_0\ge 1 \,\wedge\,Z_j\ge 1\right).&
\end{flalign*}

\subsection{Estimate $\cR_1$}
To simply notation, we will drop the index in $U_n$ from now on. Write
$$
\cR_1(q,u)= \Big|\big(\PP(Z_0=u\,\wedge\,V_\Delta^M=q-u)-\PP(Z_0=u)\PP\left(V_\Delta^M=q-u\right)\big)\Big|,
$$
which is non-vanishing only if $u\le 2K+1$.

The set $\{Z_0=u\}$ is a disjoint union of the sets
$$
Z_{\vec{v}} = \bigcap_{j=1}^uf^{-v_j}U \cap \bigcap_{i\notin \{v_j\}}f^{-i}U^c,
$$
where $\vec{v}=(v_1,\ldots, v_u)$ with $0\le v_1<\cdots<v_u\le 2K$ marks the $u$ entries to $U$ before time $2K$. Note that for $u\ge 2K+1$ (and possibly for certain $\vec{v}$ with $u\le 2K+1$), $Z_{\vec{v}}$ will be empty.

Recall that $U^i$ and $U^o$ in Assumption~\ref{a.1} are the approximations of $U$ from inside and outside, respectively. This invites us to define
$$
Z_{\vec{v}}^o = \bigcap_{j=1}^uf^{-v_j}U^o \cap \bigcap_{i\notin \{v_j\}}f^{-i}(U^i)^c,
$$
as the approximations of $Z_{\vec{v}}$ from 
outside. 
Clearly one has 
$\overline{ Z_{\vec{v}}}\subset Z^o_{\vec{v}}$ for all vectors $\vec{v}$. Moreover, there are Lipschitz functions 
$\phi^o_{\vec{v}}$ that satisfy
$$
\phi^o_{\vec{v}}(x)=\begin{cases} 
1 & x\in Z_{\vec{v}} \\
0 & x\notin Z^o_{\vec{v}}.
\end{cases}
$$
with Lipschitz constants bounded by $C_K/r_n$ for some constant $K$ depending on $f$ and $K$ (but not on $n,u$ or  $\vec{v}$), with $r_n$ as in Assumption~\ref{a.1}.\footnote{One simple way to construction such functions is to first construct Lipschitz functions on $U^{i/o}$ with norm bounded by $1/r_n$, then iteration them under $f$. Since one only need to iterate no more than $K$ times, the Lipschitz constant is affected by a constant $C_K$.} By the construction, we have
$$
\II_{Z_{\vec{v}}}\le\phi^o_{\vec{v}},
$$
with difference bounded by
$$
\int_\bM\left(\phi^o_{\vec{v}}-\II_{Z_{\vec{v}}}\right)\,d\mu\le \mu(Z^o_{\vec{v}}\setminus Z_{\vec{v}})\le\mu(U^o\setminus U)=o(\mu(U)),
$$
thanks to Assumption~\ref{a.1}.

Then we have
\begin{align*}\numberthis\label{e.zv}
&\Big|\big(\PP(Z_{\vec{v}}\,\wedge\,V_\Delta^M=q-u)-\mu(Z_{\vec{v}})\PP\left(V_\Delta^M=q-u\right)\big)\Big|\\
=&\left|\int_\bM\II_{Z_{\vec{v}}}\cdot\II_{\{V_\Delta^M=q-u\}}\,d\mu - \int_\bM\II_{Z_{\vec{v}}}\,d\mu\int_{\bM}\II_{\{V_\Delta^M=q-u\}}\,d\mu\right|\\
\le& X+Y+Z,
\end{align*}
where 
\begin{align*}
&X=\int_\bM\left(\phi^o_{\vec{v}}-\II_{Z_{\vec{v}}} \right)\II_{\{V_\Delta^M=q-u\}}\,d\mu,\\
&Y=\left|\int_\bM\phi^o_{\vec{v}}\cdot\II_{\{V_\Delta^M=q-u\}}\,d\mu - \int_\bM\phi^o_{\vec{v}}\,d\mu\int_{\bM}\II_{\{V_\Delta^M=q-u\}}\,d\mu\right|,\\
&Z=\int_\bM\left(\phi^o_{\vec{v}}-\II_{Z_{\vec{v}}}\right)d\mu\int_{\bM}\II_{\{V_\Delta^M=q-u\}}\,d\mu.
\end{align*}

Note that 
\begin{equation}\label{e.xz}
X+Z\le 2\int_{\bM}\left(\phi^o_{\vec{v}}-\II_{Z_{\vec{v}}}\right)\,d\mu = o(\mu(U)),
\end{equation}
so we are left to estimate $Y$. One can easily check that the estimates below does not depend on $M,q,u$ or $\vec{v}$.

\noindent Case 1. $f$ is non-invertible. 

In this case, we directly apply the decay of correlations~\eqref{e.decay1} for non-invertible towers to the Lipschitz function $\phi_{\vec{v}}^o$ and $L^\infty$ function $\II_{\{V_\Delta^M=q-u\}}$. This gives
\begin{equation}\label{e.y.ni}
Y\le C\|\phi_{\vec{v}}^o\|_{\Lip}\cC(\Delta) \le \frac{C_K}{r_n}\cC(\Delta).
\end{equation}

\noindent Case 2. $f$ is invertible.

We need to approximate $\II_{\{V_\Delta^M=q-u\}}$ by $L^\infty$ functions that are constant on stable disks. We take any positive integer $\Delta'\le\Delta$ and write, for $k\ge (2K+1)\Delta'$,
$$
S_k(U) =\bigcup_{i}\bigcup_{j=0}^{R_i-1}\bigcup_{\substack{\gamma\in\Gamma^s\\f^{k+j}(\gamma\cap\Omega_{0,i})\cap\partial U\ne\emptyset}}f^j(\gamma) 
$$
for the union of stable disks (and their forward images before returning to $\Omega_0$) whose image under $f^k$ will intersect with the topological boundary of $U$. Note that $f^kS_k(U)$ is a union of $f^{k+j}\gamma$ for $\gamma\in\Gamma^s$. The polynomial contraction along stable disks~\eqref{e.cont} gives
$$\diam(f^{k+j}\gamma)\le C/(k+j)^{\alpha}\le C/k^\alpha.
$$ 
If we write $B_r(\partial U)$ for the $r$-neighborhood of $\partial U$, then the observation above yields
$$
f^kS_k(U)\subset B_{C/k^{\alpha}}(\partial U).
$$
As a result, we get by the invariance of $\mu$,
$$
\mu(S_k(U))\le \mu(B_{C/k^{\alpha}}(\partial U)).
$$
Now we define (and suppress the dependence on $q,u,n,\Delta$ and $M$ for simplicity):
$$
\tilde S=\bigcup_{k=(2K+1)\Delta'}^{(2K+1)(M+\Delta'-\Delta)}S_k(U).
$$

Consider the $L^\infty$ function 
$$
\psi=\II_{\{V_{\Delta'}^{M+\Delta'-\Delta}=q-u\}}\cdot \II_{\tilde S^c}.
$$
We see that $\psi$ is constant on stable disks, since if $x\in\{V_{\Delta'}^{M+\Delta'-\Delta}=q-u\}\cap \tilde S^c$ hits $U$ under the $j$th iteration of $f$ for $j\in[(2K+1)\Delta', (2K+1)(M+\Delta'-\Delta)]$, then the entire stable disk at $x$ will be contained in $U$ under the same iteration.  Meanwhile, we can easily estimate the $L^1$ norm of the difference between $\psi$ and $\II_{\{V_{\Delta'}^{M+\Delta'-\Delta}=q-u\}}$:
\begin{align*}
\int_\bM\left|\psi- \II_{\{V_{\Delta'}^{M+\Delta'-\Delta}=q-u\}}\right|d\mu=&
\int_{\bM}1-\II_{\tilde S^c}\,d\mu\\=\,&\mu(\tilde{S})\\
\le&\sum_{k=(2K+1)\Delta'}^{(2K+1)(M+\Delta'-\Delta)}\mu(S_k(U))\\\le& \sum_{k\ge(2K+1)\Delta'}\mu(B_{C/k^{\alpha}}(\partial U)).
\end{align*}

The term $Y$ can now be estimated as
\begin{align*}
Y=&\left|\int_\bM\phi^o_{\vec{v}}\cdot\II_{\{V_{\Delta'}^{M+\Delta'-\Delta}=q-u\}}\circ f^{\Delta-\Delta'}\,d\mu - \int_\bM\phi^o_{\vec{v}}\,d\mu\int_{\bM}\II_{\{V_{\Delta'}^{M+\Delta'-\Delta}=q-u\}}\,d\mu\right|\\
\le &\left|\int_\bM\phi^o_{\vec{v}}\cdot\psi\circ f^{\Delta-\Delta'}\,d\mu - \int_\bM\phi^o_{\vec{v}}\,d\mu\int_{\bM}\psi\,d\mu\right|\\
&+2\sum_{k\ge(2K+1)\Delta'}\mu(B_{C/k^{\alpha}}(\partial U))\\
\le &\frac{C_K}{r_n}\cC(\Delta-\Delta')+2\sum_{k\ge(2K+1)\Delta'}\mu(B_{C/k^{\alpha}}(\partial U))\numberthis\label{e.y}
\end{align*}
for any $0\le\Delta'\le\Delta<N'$.

Collect~\eqref{e.zv},~\eqref{e.xz} and~\eqref{e.y} (or~\eqref{e.y.ni} in the non-invertible case) and sum over $u$ and $\vec{v}$, we get (recall that we are only interested in the case $u\le 2K+1$, since otherwise $\{Z_0=u\}$ will be empty; therefore the total number of summands is bounded by a constant that depends on $K$):
\begin{align*}
\cR_1\le& \sup_{q,M}\sum_{u=1}^{q-1}\sum_{\substack{\vec{v}=(v_1,\ldots, v_u),\\0\le v_1<\cdots<v_u\le 2K}} \Big|\big(\PP(Z_{\vec{v}}\,\wedge\,V_\Delta^M=q-u)-\mu(Z_{\vec{v}})\PP\left(V_\Delta^M=q-u\right)\big)\Big|\\
\le& C_K'\Bigg(o(\mu(U_n))+\frac{1}{r_n}\cC(\Delta-\Delta')+\sum_{k\ge(2K+1)\Delta'}\mu(B_{C/k^{\alpha}}(\partial U_n))\Bigg), \numberthis\label{e.r1}
\end{align*}
where $C_K'$ is a constant that does not depend on $\Delta, \Delta'$ or $U$. The last term does not show up when $f$ is non-invertible.

\subsection{Estimate $\cR_2$}\label{s.6.2}
We use a strategy similar to the proof of Theorem~\ref{m.4}. Recall that $\mu_0$ is the measure supported on $\Omega_0$ and is invariant under $T=f^R$, such that $\mu$ is the lift of $\mu_0$ given by
$$
\mu(B) = \sum_{i=1}^\infty\sum_{k=0}^{R_i-1}\mu_0(f^{-k}(B)\cap \Omega_{0,i}).
$$
In particular, we have
\begin{align*}
\mu (U\cap f^{-j}(U))=&\sum_{i=1}^\infty\sum_{k=0}^{R_i-1}\mu_0(\Omega_{0,i}\cap f^{-k}U\cap f^{-(k+j)}U).
\end{align*}

\subsubsection{Case 1. $f$ is non-invertible}
Recall that $T=f^R$ is the induced map on $\Omega_0$. For each vector $\vec{i}_l = (i_1,i_2,\ldots,i_{l})\in\NN^l$, we define the {\em $l$-cylinder} $I_{\vec{i}_l}$ to be 
$$
I_{\vec{i}_l} = \Omega_{0,i_1}\cap T^{-1}\Omega_{0,i_2}\cap\cdots\cap T^{-l}\Omega_{0,i_{l}}
$$

We are particularly interested in those cylinders $I_{\vec{i}_l}$ where the second last visited partition element $\Omega_{0,i_{l-1}}$ has a short return time $R_{i_{l-1}}$. To be more precise, for each integer $s>0$, we define the collection of `good' cylinders to be 
$$
\cI^{\cG}_s=\{I_{\vec{i}_l}: R_{i_{l-1}}\le s\}.
$$
$\cI^{\cG}_s$ consists of all $l$-cylinders where the travel time from $\Omega_{0,i_{l-1}}$ to $\Omega_{0,i_{l}}$ is less than $s$.
We also write, for each $k$ large enough, the collection of `good' cylinders whose length (under the iteration of $f$) is around $k$: 
$$
\cI^\cG_s(k) = \{I_{\vec{i}_l}\in \cI^\cG_s: \sum_{j=1}^{l-1}R_{i_j}\le k < \sum_{j=1}^{l} R_{i_j}\}.
$$
For each vector $\vec{i}_l = (i_1,\ldots,i_l)$ such that $I_{\vec{i}_l}\in\cI^\cG_s(k+j)$, we write 
$$\vec{i}'_l = (i_1,\ldots,i_{l-1}),$$
i.e., we drop the last component from $\vec{i}_l$. Then the $(l-1)$-cylinder $I_{\vec{i}_l'}$ contains $I_{\vec{i}_l}$, and satisfies 
\begin{enumerate}
	\item $R_{i_{l-1}}<s$;
	\item $\sum_{j=1}^{l-1}R_{i_j}\le k+j < \sum_{j=1}^{l} R_{i_j}$.
\end{enumerate}
For given $k<R_i$, $j>K_0$ and $s=s(k+j)$ given by Assumption~\ref{a.3}, we denote by $\tilde{\Omega}_i$ be the union of all the `good' $(l-1)$-cylinders in $\Omega_{0,i}$:
\begin{equation}\label{e.tilde}
\tilde{\Omega}_i = \bigcup_{I_{\vec{i}_l}\in\cI^\cG_s(k+j),I_{\vec{i}_l}\subset \Omega_{0,i} } I'_{\vec{i}_l}
\end{equation}

The next lemma is similar to the distortion estimate in the proof of Theorem~\ref{m.4}.

\begin{lemma}\label{l.dist} We have
$$
\mu_0(\tilde{\Omega}_i\cap f^{-k}U\cap f^{-(k+j)}U)\lesssim\mu(U) \mu_0 \left(\bigcup_{\substack{I'_{\vec{i}_l}:I_{\vec{i}_l}\in\cI^\cG_{k+j},I'_{\vec{i}_l}\cap \Omega_{0,i}\cap f^{-k}U\ne\emptyset }} I'_{\vec{i}_l}\right).
$$
\end{lemma}
\begin{proof}
We have
\begin{align*}
&\mu_0(\tilde{\Omega}_{i}\cap f^{-k}U\cap f^{-(k+j)}U)\\
\le&\sum_{\substack{I'_{\vec{i}_l}:I_{\vec{i}_l}\in\cI^\cG_{k+j},I'_{\vec{i}_l}\cap \Omega_{0,i}\cap f^{-k}U\ne\emptyset }}\mu_0(f^{-(k+j)}U\cap I'_{\vec{i}_l})  \\
=&	\sum_{\substack{I'_{\vec{i}_l}:I_{\vec{i}_l}\in\cI^\cG_{k+j},I'_{\vec{i}_l}\cap \Omega_{0,i}\cap f^{-k}U\ne\emptyset }}\frac{\mu_0(f^{-(k+j)}U\cap I'_{\vec{i}_l})}{\mu_0(I'_{\vec{i}_l})} \mu_0(I'_{\vec{i}_l})  \\
\lesssim& \sum_{\substack{I'_{\vec{i}_l}:I_{\vec{i}_l}\in\cI^\cG_{k+j},I'_{\vec{i}_l}\cap \Omega_{0,i}\cap f^{-k}U\ne\emptyset }}\frac{\mu_0\left(T^{l-1}\left(f^{-(k+j)}U\cap I'_{\vec{i}_l}\right)\right)}{\mu_0(T^{l-1}I'_{\vec{i}_l})} \mu_0(I'_{\vec{i}_l}) ,
\end{align*}
where we used the distortion estimate on the last inequality.

Note that the denominator satisfies $\mu_0(T^{l-1}I'_{\vec{i}_l}) = \mu_{0}(\Omega_0)=1$. For the numerator, we write, with $b= k+j-\sum_{m=1}^{l-1} R_{i_m}\in [0,s]\cap\NN$,
\begin{align*}
\mu_0\left(T^{l-1}\left(f^{-(k+j)}U\cap I'_{\vec{i}_l}\right) \right)\le&\, \mu_0\left(T^{l-1}\left(f^{-(k+j)}U\right)\right)\\
=& \mu_0\left(f^{\sum_{m=1}^{l-1} R_{i_m}-(k+j)}U\right)\\
=& \mu_0(f^{-b}U)\\
\le & C_0 \mu(f^{-b}U) = C_0 \mu(U)
\end{align*}
for some constant $C_0>0$ independent of $b$.\footnote{The existence of such $C_0$ follows from the facts that $\mu_0 = \mu|_{\Omega_0}$ and $\mu_0(R)<\infty$.}

Now we conclude that
$$
\mu_0(\tilde{\Omega}_{i}\cap f^{-k}U\cap f^{-(k+j)}U)\lesssim\mu(U) \mu_0 \left(\bigcup_{\substack{I'_{\vec{i}_l}:I_{\vec{i}_l}\in\cI^\cG_{k+j},I'_{\vec{i}_l}\cap \Omega_{0,i}\cap f^{-k}U\ne\emptyset }} I'_{\vec{i}_l}\right).
$$
\end{proof}


Note that if $I'_{\vec{i}_l}$ is a `good' $(l-1)$-cylinder that has intersection with $f^{-k}U$, then the backward contraction along unstable disks~\eqref{e.exp} gives $$\diam( I'_{\vec{i}_l})\lesssim (j-s)^{-\alpha}.$$
As a result, such cylinders must be contained in the $(j-s)^{-\alpha}$-neighborhood of $\partial U$.
This together with the previous lemma and Assumption~\ref{a.3} give:

\begin{align*}
&\mu (U\cap f^{-j}(U))\\\le &\sum_{i}\sum_{k=0}^{R_i-1}\mu_0(\tilde{\Omega}_i\cap f^{-k}U\cap f^{-(k+j)}U)+\sum_i\sum\limits_{k=0}^{R_i}\mu_0(f^{-k}U_n\cap (\Omega_{0,i}\setminus \tilde{\Omega}_i))\\
\lesssim&\mu(U)\sum_{i}\sum_{k=0}^{R_i-1}\sum_{ \substack{I'_{\vec{i}_l}:I_{\vec{i}_l}\in\cI^\cG_{k+j},I'_{\vec{i}_l}\cap \Omega_{0,i}\cap f^{-k}U\ne\emptyset }}
\mu_0(I'_{\vec{i}_l})+G(j)\mu(U)\\
\le &\mu(U)\left(\mu\left(U\cup B_{(j/2)^{-\alpha}}(\partial U)\right)+k^{-p''}\right)\\
\lesssim &\mu(U)\left(\mu(U)+j^{-\alpha p'}+j^{-p''}\right)
\end{align*}

The rest of the proof follows the lines of the proof of Theorem~\ref{m.4}, with $\kappa_n$ replaced by $\Delta$. We obtain, for any $K'< K$,
\begin{align*}
\cR_2\lesssim (2K+1)\mu(U)\left(\Delta K\mu(U)+K^{-\min\{\alpha p',p''\}+1}+K\mu(U_n)+(K')^{-(p'-1)}+\frac{K'}{K}\right).
\end{align*}

\subsubsection{Case 2. $f$ is invertible}
We define the cylinders $T_{\vec{i}_l}$ and the collection of `good' cylinders $\cI^\cG_{s}(k)$ in the same way as before. We will estimate each set $\Omega_{0,i}\cap f^{-k}U\cap f^{-(k+j)}U$ using the conditional measures of $\mu_0$. 

Recall that $\mu_\gamma$ are the conditional measures of $\mu_0$ for $\gamma\in\Gamma^u$. Similar to the proof of Lemma~\ref{l.dist}, we have
\begin{align*}
&\mu_\gamma(\tilde{\Omega}_{i}\cap f^{-k}U\cap f^{-(k+j)}U)\\
\le&\sum_{\substack{I'_{\vec{i}_l}:I_{\vec{i}_l}\in\cI^\cG_{k+j},I'_{\vec{i}_l}\cap \Omega_{0,i}\cap f^{-k}U\ne\emptyset }}\mu_\gamma(f^{-(k+j)}U\cap I'_{\vec{i}_l})  \\
=&	\sum_{\substack{I'_{\vec{i}_l}:I_{\vec{i}_l}\in\cI^\cG_{k+j},I'_{\vec{i}_l}\cap \Omega_{0,i}\cap f^{-k}U\ne\emptyset }}\frac{\mu_\gamma(f^{-(k+j)}U\cap I'_{\vec{i}_l})}{\mu_\gamma(I'_{\vec{i}_l})} \mu_\gamma(I'_{\vec{i}_l})  \\
\lesssim& \sum_{\substack{I'_{\vec{i}_l}:I_{\vec{i}_l}\in\cI^\cG_{k+j},I'_{\vec{i}_l}\cap \Omega_{0,i}\cap f^{-k}U\ne\emptyset }}\frac{\mu_{\tilde{\gamma}}\left(T^{l-1}\left(f^{-(k+j)}U\cap I'_{\vec{i}_l}\right)\right)}{\mu_{\tilde{\gamma}}(T^{l-1}I'_{\vec{i}_l})} \mu_\gamma(I'_{\vec{i}_l}) ,
\end{align*}
where $\tilde{\gamma} = \tilde{\gamma}(I'_{\vec{i}_l}) = \gamma(T^{l-1}x)$ for $x\in\gamma\cap I'_{\vec{i}_l}$. As before, the denominator is bounded from above, and the numerator satisfies
$$
\mu_{\tilde{\gamma}}\left(T^{l-1}\left(f^{-(k+j)}U\cap I'_{\vec{i}_l}\right)\right)\le \mu_{\tilde{\gamma}}(f^{-b}U)\lesssim \mu(U)
$$
with $b= k+j-\sum_{m=1}^{l-1} R_{i_m}\in [0,s]\cap\NN$, and the last inequality follows from Assumption~\ref{a.4}. It then follows that (with $m(\gamma)$ the transverse measure): 
\begin{align*}
&\mu (U\cap f^{-j}(U))\\\le& \int\sum_{i}\sum_{k=0}^{R_i-1}\mu_\gamma(\tilde{\Omega}_i\cap f^{-k}U\cap f^{-(k+j)}U)\,dm(\gamma)+\sum_i\sum\limits_{k=0}^{R_i}\mu_0(f^{-k}U_n\cap (\Omega_{0,i}\setminus \tilde{\Omega}_i))\\
\lesssim&\mu(U)\int\sum_{i}\sum_{k=0}^{R_i-1}\sum_{ \substack{I'_{\vec{i}_l}:I_{\vec{i}_l}\in\cI^\cG_{k+j},I'_{\vec{i}_l}\cap \Omega_{0,i}\cap f^{-k}U\ne\emptyset }}
\mu_0(I'_{\vec{i}_l})\,dm(\gamma)+G(j)\mu(U)\\
\le &\mu(U)\left(\int\mu_\gamma\left(U\cup B_{(j/2)^{-\alpha}}(\partial U)\right)\,dm(\gamma)+k^{-p''}\right)\\
= &\mu(U)\left(\mu\left(U\cup B_{(j/2)^{-\alpha}}(\partial U)\right)+k^{-p''}\right)\\
\lesssim &\mu(U)\left(\mu(U)+j^{-\alpha p'}+j^{-p''}\right).
\end{align*}

\subsection{Collect the estimates}
In the non-invertible case, we have 
\begin{align*}
&\left|\PP(V_0^{N'}=k)- m(\{k\})\right|\le CN'(\cR_1+\cR_2)+\Delta\mu(U)\\
\lesssim& Ko(1)+\frac{C_K\cC(\Delta)}{Kr_n\mu(U_n)}+\Delta K\mu(U_n)+K^{-\min\{\alpha p',p''\}+1}+K\mu(U_n)\\&+(K')^{-\min\{\alpha p',p''\}+1}+\frac{K'}{K}+\Delta\mu(U_n).
\end{align*}

Sending $n$ to infinity then $K$ to infinity with  $K'=\sqrt{K}$, we see that the error term goes to zero as $n$ goes to infinity, provided that 
$$
r_n = o\left(\frac{\mu(U)}{\cC(\Delta/2)}\right).
$$
This will also make the compound binomial distribution converging to the compound Poisson distribution, following Remark~\ref{r.2}. This finishes the proof of Theorem~\ref{m.5} in the non-invertible case.

In the invertible case,
\begin{align*}
&\left|\PP(V_0^{N'}=k)- m(\{k\})\right|\le CN'(\cR_1+\cR_2)+\Delta\mu(U)\\
\lesssim& C_K\left(o(1)+\frac{\cC(\Delta)}{r_n\mu(U_n)}+\frac{1}{\mu(U_n)}\sum_{k\ge(2K+1)\Delta'}\mu(B_{C/k^{\alpha}}(\partial U_n))\right)+\Delta K\mu(U_n)\\&+K^{-\min\{\alpha p',p''\}+1}+K\mu(U_n)+(K')^{-\min\{\alpha p',p''\}+1}+\frac{K'}{K}+\Delta\mu(U_n).
\end{align*}

The third term is estimated by Assumption~\ref{a.2}. We take $\Delta = \mu(U)^{-1+\varepsilon}$ for $\varepsilon>0$ small enough, then
\begin{align*}
&\frac{1}{\mu(U_n)}\sum_{k\ge(2K+1)\Delta'}\mu(B_{C/k^{\alpha}}(\partial U_n))\\
\lesssim &\frac{1}{\mu(U_n)}\sum_{k\ge(2K+1)\Delta'}k^{-\alpha p'} \\
\lesssim &C_K\mu(U_n)^{-1}\Delta^{-(\alpha p'-1)}\\
=&C_K\mu(U_n)^{(1-\varepsilon)(\alpha p'-1)-1},
\end{align*}
which vanishes as long as $\alpha p'>2$ and $\varepsilon$ is taken small enough. We conclude the proof of Theorem~\ref{m.5} in the invertible case, and Theorem~\ref{m.2},~\ref{m.1} follows.

\subsection{Synchronizing $K$ and $n$}
The following proposition is a by-product from the proof of the previous theorem, which states that $\beta_\ell$ defined in~\eqref{e.beta} by synchronizing $K$ and $n$ is, in fact, the same as $\hat{\alpha}_\ell$. 

\begin{proposition}\label{p.2}
Under the assumptions of Theorem~\ref{m.5}, for any increasing sequence  $\{s_n\}$ with $s_n\to\infty$ and $s_n\mu(U_n)\to 0$, the sequence $\{\beta_\ell\}$ defined by~\eqref{e.beta} exists and satisfies
$
\beta_\ell=\hat{\alpha}_\ell
$
for all $\ell \ge 1$.
\end{proposition}

The proof follows the estimate on $\cR_2$ and is extremely similar to the proof of Proposition~\ref{p.1}, therefore will be omitted.

\subsection{Proof of Corollary~\ref{c.1}}
Assume that $\pi_{\ess}(U_n)\to\infty$. By Lemma~\ref{l.3.1} we have $\hat{\alpha}_1 = 1$ and $\hat{\alpha}_\ell =0$ for all $\ell\ge 2$. Then Theorem~\ref{t.1} gives
$$
\alpha_1 = \hat{\alpha}_1 - \hat\alpha_{2} = 1,\mbox{ and } \alpha_\ell = 0\mbox{ for all }\ell \ge 2.
$$
As a result, we have 
$$
\lambda_1 = \frac{\alpha_1 - \alpha_2}{\alpha_1} = 1,\mbox{ and } \lambda_\ell = 0\mbox{ for all }\ell \ge 2.
$$

In view of Theorem~\ref{m.3}, we only need to show that the entry times distribution $\zeta_{U_n}$ converges to a Poisson distribution with parameter $\tau>0$, where $\tau$ is given by~\eqref{e.1}. For this purpose, we apply Theorem~\ref{m.5} (or Theorem~\ref{m.4} and Theorem~\ref{t.4} when the tower is defined using the first return map). In this case, the compound part is a trivial distribution with $\PP(Z_j=1) = 1$. Then the compound Poisson distribution reduces to a Poisson distribution with parameter $\alpha_{1}\tau = \tau$. 

To finish the proof, we state a general proposition regarding the periodic of $U_n$. Note that the proof does not require the system to be measure preserving, and $\Lambda$ need not have zero measure.

\begin{proposition}\label{p.3}
Let $f$ be a continuous map on the compact metric space $\bM$, and $\{U_n\}$ a nested sequence of sets (need not be open), such that $\cap_nU_n = \cap_n \overline{U_n}$.  Then $\pi(U_n) \to\infty$ if and only if $\Lambda=\cap_nU_n $ does not contain periodic point and is contained in a fundamental domain of $f$, i.e., $\Lambda$ intersects every orbit at most once.
\end{proposition}

\begin{proof}
We first prove the `only if' part. Assume there exists a point $x$ such that $\Lambda\cap\Orb(x)$ contains two points $y$ and $y'$ (if $y=y'$ then we are in the periodic point case). Without loss of generality, we take $k>0$ such that $y' = f^k (y)$. Then we have
$\pi(U_n)\le k$ for every $n$ since $y\in U_n\cap f^{-k} U_n,$ a contradiction.

For the `if' part, we prove by contradiction. First, observe that $\pi(\cdot)$ is monotonic, i.e., $\pi(U)\ge \pi(V)$ if $U\subset V$. Therefore, if the sequence $\pi(U_n)$ does not go to infinity, it must remain bounded, thus has to converge to a finite number $N$. 

It then follows that for each $n$ large enough, there exists $x_n\in U_n$ such that $f^N(x_n)\in U_n$. Take a subsequence if necessary, we may assume that $x_n\to x$. Note that for each $n$, we have $x\in \overline{U_n}$. This shows that $x\in\Lambda=\cap_n\overline{U_n}$. Since $f$ is continuous, $f^N(x_n)\to f^N(x)$, which must be contained in $\Lambda$ according to the same argument. Then either $\Lambda\cap\Orb(x)$ contains at least two points, or $x = f^N(x)$, which means $x$ is a periodic points; both cases contradict with the assumption that $\Lambda$ does not contain periodic points and is contained in a fundamental domain of $f$.
 
\end{proof}
Note that in our setting, the condition $\cap_nU_n = \cap_n \overline{U_n}$ holds from the construction. Now the final statement of Corollary~\ref{c.1} follows from the observation that if $\Lambda$ is contained in a fundamental domain without periodic point, then $\pi(U_n)\to\infty$, which means $\pi_{\ess}(U_n)\to\infty$. We conclude the proof of Corollary~\ref{c.1}.

\end{document}